\documentclass[11pt, a4paper, oneside, final, reqno]{amsart}

\usepackage{amssymb}

\newcommand{\calA}{\mathcal{A}}
\newcommand{\calB}{\mathcal{B}}

\newcommand{\calL}{\mathcal{L}}

\newcommand{\bbN}{\mathbb{N}}
\newcommand{\bbZ}{\mathbb{Z}}

\newcommand{\bbC}{\mathbb{C}}
\newcommand{\bbT}{\mathbb{T}}

\newcommand{\Hom}{\operatorname{Hom}}
\newcommand{\CP}{\operatorname{CP}}
\newcommand{\UCP}{\operatorname{UCP}}
\newcommand{\CB}{\operatorname{CB}}
\newcommand{\QC}{\operatorname{QC}}
\newcommand{\TC}{\operatorname{TC}}

\newcommand{\id}{\operatorname{id}}

\newcommand{\Tr}{\operatorname{Tr}}
\newcommand{\Choi}{\operatorname{Choi}}
\newcommand{\mk}{\operatorname{mk}}

\newcommand{\Cred}{C^*_r}

\usepackage{thmtools}

\theoremstyle{definition}
\newtheorem{definition}{Definition}[section]

\theoremstyle{plain}
\newtheorem{theorem}[definition]{Theorem}
\newtheorem{proposition}[definition]{Proposition}
\newtheorem{corollary}[definition]{Corollary}
\newtheorem{lemma}[definition]{Lemma}

\theoremstyle{remark}
\newtheorem{remark}[definition]{Remark}
\newtheorem{example}[definition]{Example}

\numberwithin{equation}{section}
\usepackage{hyperref}
\hypersetup{pdftitle={}}
\usepackage[nameinlink, noabbrev, capitalize]{cleveref}
\usepackage[latin1]{inputenc}
\usepackage[backend=biber,style=alphabetic, sorting=nyt, maxnames=50]{biblatex}
\usepackage{xcolor}
\usepackage[all]{xy}

\usepackage[a4paper, total={6.5in, 10in}]{geometry}
\usepackage{microtype}

\title[]{Metrics on completely positive maps via noncommutative geometry}

\date{February 6, 2026}

\author[1]{Are Austad}
\address{Are Austad, Department of Mathematics, University of Oslo, P.O.Box 1053 Blindern, 0316 Oslo, Norway}
\email{areaus@math.uio.no}

\author[2]{Erik B\'edos}
\address{Erik B\'edos, Department of Mathematics, University of Oslo, P.O.Box 1053 Blindern, 0316 Oslo, Norway}
\email{bedos@math.uio.no}

\author[3]{Jonas Eidesen}
\address{Jonas Eidesen, Department of Mathematics, University of Oslo, P.O.Box 1053 Blindern, 0316 Oslo, Norway}
\email{jonaeid@math.uio.no}

\author[4]{Nadia S.~Larsen}
\address{Nadia S.~Larsen, Department of Mathematics, University of Oslo, P.O.Box 1053 Blindern, 0316 Oslo, Norway}
\email{nadiasl@math.uio.no}

\author[5]{Tron Omland}
\address{Tron Omland, Norwegian National Security Authority (NSM)
\and Department of Mathematics, University of Oslo, P.O.Box 1053 Blindern, 0316 Oslo, Norway}
\email{tron.omland@gmail.com}



\begin{document}

\begin{abstract}
    We study methods of inducing metrics on unital completely positive maps by employing seminorms arising in noncommutative geometry. Our main approach relies on the development of an infinite-dimensional $C^*$-algebraic analogue of the Choi-Jamio\l{}kowski isomorphism. Under suitable conditions, we show that the induced metrics satisfy the quantum information theoretic properties of stability and chaining. Moreover, we show how to generate such metrics using constructions native to noncommutative geometry, by for example using external Kasparov products of spectral triples.
\end{abstract}

\maketitle

\section{Introduction}

Trace-preserving completely positive maps are ubiquitous in quantum information theory as, among others, the means to model evolution of quantum systems, these generally being described by matrix algebras. Through a standard duality perspective, there is a companion class of unital, completely positive maps, and these in turn are prominent in operator algebras and noncommutative geometry, especially in relation to deep questions about approximation properties of operator algebras (in an infinite-dimensional setting).  Among many questions which have been asked about trace-preserving completely positive maps one finds those about existence of a metric structure, or distance, capable of capturing permanence properties such as chaining, a form of coherent composition, or stability under tensorial expansion of the system with an ancilla system present, cf. \cite{Gilchrist, BorovykVellambi25}. In this work we propose a framework for developing a metric structure on the set of (unital) completely positive maps between two, not necessarily finite-dimensional $C^*$-algebras, by employing constructions native to noncommutative geometry.

Attempts to develop a metric structure have so far concentrated on the finite-dimensional setting, namely matrix algebras \cite{Gilchrist, BorovykVellambi25}, or, in the operator algebraic framework, on the von Neumann algebraic setting, where normal states are a core ingredient \cite{AnshuJekelLandry25, AraizaJungeWu25, DawsQuantumGraphs25}. An exception to this is the Wasserstein type metric of order one on unital completely positive maps between tensor products of unital $C^*$-algebras introduced in \cite{DM23}. In a different direction we find work on the Bures distance for completely positive maps between $C^*$-algebras from several authors, see \cite{KretschmannSchlingemannWerner2008,BhatSumesh2013}.

The main thrust of the present work is the formalization and study of two methods of inducing metrics on the set of unital completely positive maps  between $C^*$-algebras through the use of seminorms, providing a rich source of examples coming from noncommutative geometry.  In particular, it provides a link to the theory of compact quantum metric spaces, as pioneered by Connes and Rieffel, see, e.g., \cite{Connes1989, Rieffel98, RieffelMetricsOnStateSpaces1999, Rieffel02, Rieffel2004}. Very satisfyingly, we find that both stability and chaining have natural and precise formulations and, moreover, are in place for large classes of examples. 
Questions regarding metrics having stability and chaining are not standard in the context of quantum metrics in noncommutative geometry, and we hope that this work opens up a new fruitful line of interaction with quantum information theory. 

The investigation of these questions in such a large generality necessitates the development of novel technical tools not previously found in the literature. Of note, we develop a $C^*$-algebraic Choi-Jamio\l{}kowski type embedding to account for the fact that unlike matrix algebras, general $C^*$-algebras are neither nuclear nor necessarily isomorphic to their opposite algebras.

Before we delve into the details of the paper, we take a moment to briefly describe the origins of the notions of stability and chaining. In \cite{Gilchrist},  see also \cite{BorovykVellambi25}, the authors outline a program to define distances, or metrics, on sets of quantum processes, in particular covering \emph{quantum channels}, i.e., trace-preserving completely positive maps between matrix algebras $A$ and $B$. Several potentially meaningful distances are put forward, along with a list of  desirable properties these distances should satisfy.  
The first of two particularly important properties a (possibly extended) metric $d$ should satisfy is \emph{stability}, in the sense that if $\mathcal{F}, \mathcal{G}$ are quantum channels from $A$ to $B$, then $d(\id_n \otimes \mathcal{F}, \id_n \otimes \mathcal{G}) = d(\mathcal{F},\mathcal{G})$ must hold for every $n \in \bbN$. Here $\id_n \otimes \mathcal{F}$ and $\id_n \otimes \mathcal{G}$ are amplifications of the initial channels acting between  $M_n(\bbC) \otimes A$ and $M_n(\bbC) \otimes B$. From a quantum information theoretic perspective, this property is supposed to reflect the fact that an unrelated ancillary system should not affect the distance between initial quantum channels modelling quantum processes.  In \cite{Gilchrist} and \cite{BorovykVellambi25} it is  expected that an analogue of $d$ exists on all amplifications  for every $n \in \bbN$.

The second key property a distance on quantum channels should feature is \emph{chaining}, which aims to capture the idea that if one divides a process into smaller steps, the total error should be smaller than the sum of the stepwise errors. Here one interprets the distance between quantum processes as a measure of the error between an ideal process and an observed process. Roughly, this is expressed as $d(\mathcal{G}_1\circ \mathcal{F}_1, \mathcal{G}_2\circ \mathcal{F}_2)\leq d(\mathcal{G}_1, \mathcal{G}_2)+d(\mathcal{F}_1,\mathcal{F}_2)$, assuming the quantities are meaningfully defined (see \cref{def:chaining}).

Among the metrics highlighted in \cite{Gilchrist} as the most promising, especially with regard to stability and chaining,  there are one arising from a stabilizing procedure and one based on the Choi-Jamio\l{}kowski isomorphism. The more recent overview paper \cite{BorovykVellambi25} reinforces this view, while also presenting some new definitions of distances. 

Our concern here is with metrics on sets of completely positive maps. The set $\UCP(A, B)$ of unital completely positive maps between unital $C^*$-algebras $A$ and $B$ is known to be dual to the class of quantum channels from $B$ to $A$ (see for instance \cite{JKP2009, Wat2009}). Thus we translate the discussions in \cite{Gilchrist} and \cite{BorovykVellambi25} to the setting of $\UCP(A, B)$ whenever relevant and without further mention. 
In this spirit, we put forward definitions of stability and chaining in an infinite-dimensional setting in \cref{sec:discussion-inducing-metrics}. 
We will establish these properties in two distinct directions: one involves a stabilizing process similar to \cite{Gilchrist, BorovykVellambi25}, but with the difference that it starts from an extended metric $d_L$ arising from a  seminorm $L$ on a unital $C^*$-algebra $A$, see \emph{pullback-induced extended metrics} in \cref{def:UCP metric D} and \cref{UCP}. 
In the main results of \cref{UCP}, we
find sufficient conditions under which stability and chaining are satisfied for these extended metrics, cf.~\cref{Ln2} and \cref{cor:chaining-for-pullback-metric}, respectively. 
 
Another way of inducing an extended metric on $\UCP(A,B)$ motivated by \cite{Gilchrist} involves embedding $\UCP(A,B)$ into the state space of a higher-dimensional $C^*$-algebra using the Choi-Jamio\l{}kowski isomorphism. If one naively assumes that such an embedding exists in general, say $\omega \colon \UCP(A,B) \to S(A \otimes_{\rm max} B^{\rm op})$ is such an embedding (in analogy with the constructions from \cite{Gilchrist, BorovykVellambi25}), then, given a seminorm $L$ on $A \otimes_{\rm max} B^{\rm op}$, one can define an \emph{embedding-induced extended metric} 
\begin{equation*}
    \Delta_{L}(F,G) := \mk_L(\omega(F),\omega(G))\text{ for }F,G\in \UCP(A,B),
\end{equation*}  
where $\mk_L$ is the Monge-Kantorovi\v{c} extended metric associated with $L$. Monge-Kantorovi\v{c} metrics are central to noncommutative geometry, see \cref{sec:lipschitz-seminorms-and-cqms} for a short introduction. However, such a Choi-Jamio\l{}kowski embedding is not known to exist for general $C^*$-algebras. We develop such an embedding in \cref{sec:tracial-Cstar-algebras-and-cp-maps} where we consider a $C^*$-algebra $B$ with a faithful trace $\tau$, and we define the map $\omega_\tau \colon \Hom(A,B) \to \Hom(A \odot B^{\rm op},\bbC)$ which gives an infinite-dimensional analogue of the map appearing in the Choi-Jamio\l{}kowski isomorphism. This map always extends to give a well-defined map into $\Hom(A \otimes_{\rm max} B^{\rm op},\bbC)$, under which $\UCP(A,B)$ embeds into $S(A \otimes_{\rm max} B^{\rm op})$, cf.~\cref{thm:maps with states as associated functionals}. In fact, \cref{thm:maps with states as associated functionals} precisely characterizes which maps from $A$ to $B$ that correspond to states on $A \otimes_{\rm max} B^{\rm op}$. We call these maps \emph{trace channels}, see \cref{def:trace channels}, and they will be the main object of study in \cref{sec:tracial-Cstar-algebras-and-cp-maps} and \cref{sec:induced-metrics-between-cp-maps}. 

A significant part of the present work consists of formalizing an infinite-dimensional analogue of the Choi-Jamio\l{}kowski map for $C^*$-algebras. Related work has appeared in several recent papers: In the quantum graph literature an analogue appears for finite-dimensional $C^*$-algebras, cf.~\cite{Wasilewski2024,CourtneyGanesanWasilewski25}. In a von Neumann algebraic setting, there exist some infinite-dimensional versions of the Choi-Jamio\l{}kowski isomorphism, see \cite{Haap2021, HKS2024} and references therein.

The technical difficulty one encounters when moving from a von Neumann algebraic setup to a $C^*$-algebraic setting pertains mainly to the need to work with state spaces of tensor products of $C^*$-algebras and their opposite algebras. This difficulty does not appear in the finite-dimensional setting as $M_n(\bbC)$ is nuclear and isomorphic to its opposite algebra.  
Neither of these facts are true for general $C^*$-algebras. 
Thus, without requiring nuclearity of the $C^*$-algebras in question, certain maps and their continuity properties very much depend on whether one employs the minimal or the maximal tensor product. A celebrated result of Kirchberg, cf.~\cite{Kir93,Kir94}, generalizing a landmark result of Connes, cf.~\cite{Connes1976}, establishes continuity with respect to the minimal tensor product of a natural map $\mu_\tau$, cf.~\cref{thm:Kirchberg's Theorem for continuity of maximally entangled state}. The map $\mu_\tau$ is dependent on a trace $\tau$, and Kirchberg's Theorem states that $\mu_\tau$ is continuous in the spatial norm if and only if the corresponding trace $\tau$ is amenable. This allows us to define a tractable version of the Choi-Jamio\l{}kowski map for maps between $C^*$-algebras, where the target $C^*$-algebra is unital and admits an amenable trace, cf.~\cref{thm:maps with states as associated functionals}. 

We have chosen to formulate our definition of a Choi-Jamio\l{}kowski map for $C^*$-algebras in the most general case available to us, although admittedly the scenario of non-nuclear $C^*$-algebras is perhaps of little appeal, or apparent relevance, to questions in quantum information theory. Nevertheless, we have done so in anticipation of further interaction between the fields.

We are then able to construct extended metrics on the space of trace channels by employing techniques from noncommutative geometry. 
In the main results of \cref{sec:induced-metrics-between-cp-maps} we find, under reasonable compatibility assumptions, sufficient conditions for stability and chaining to hold, see \cref{thm:stability 1,thm:chaining 1}.

In \cref{sec:applications}, we illustrate their validity in naturally occurring examples by once more appealing to constructions in noncommutative geometry. In fact, the main result of \cref{sec:applications}, \cref{thm:stability 2}, unexpectedly yields that by iteratively employing the external Kasparov product of spectral triples, we may generate stable sequences of extended metrics. 
Furthermore, the sequences of extended metrics constructed using \cref{thm:stability 2} reflect the physical interpretation that the internal structure of an ancillary system does not affect quantum processes in the system under consideration. It is then altogether  surprising that this property is guaranteed by the external Kasparov product.
 
Moreover, we prove that chaining is satisfied for extended metrics arising from spectral triples coming from proper length functions on countable amenable groups if we restrict the unital completely positive maps in question to those arising as Fourier multipliers, cf.~\cref{thm:chaining 3}.

A large class of examples exhibiting the best possible behavior where chaining and stability hold simultaneously is provided by countable amenable groups with natural length functions, see \cref{example:group-example-chaining-and-stability}. The construction of seminorms coming from spectral triples is of importance in the study of compact quantum metric spaces in noncommutative geometry. It has been investigated in a variety of different settings, e.g.~for group $C^*$-algebras \cite{OzawaRieffel, ChristRieffel}, crossed products \cite{HawkinsSkalskiWhiteZacharias, Klisse2024, AustadKaadKyed}, and quantum groups \cite{BhowmickVoigtZacharias, AguilarKaad2018, KaadKyed2025, AustadKyed2025}, though this is far from a complete list. In a related line of research, \cite{AnshuJekelLandry25} investigates seminorms (which a priori do not come from spectral triples) on quantum permutation groups for the purpose of studying the quantum Wasserstein distance.

The structure of the paper is as follows. We collect a number of important notions and set  notation in \cref{sec:prelims}. In \cref{sec:discussion-inducing-metrics} we formalize the notions of stability and chaining mathematically and discuss pullback-induced and embedding-induced extended metrics on sets of unital completely positive maps. We study stability and chaining for pullback-induced extended metrics in \cref{UCP}. In \cref{sec:tracial-Cstar-algebras-and-cp-maps} we study the $C^*$-algebraic Choi-Jamio\l{}kowski map and use it to introduce the notion of trace channels. This map is then used in \cref{sec:induced-metrics-between-cp-maps} to define and study embedding-induced extended metrics on the set of trace channels between $C^*$-algebras. We especially focus on criteria guaranteeing that these extended metrics satisfy stability and chaining. Lastly, in  \cref{sec:applications} we use machinery from noncommutative geometry to show that the necessary conditions for stability and chaining from \cref{sec:induced-metrics-between-cp-maps} are automatically satisfied for extended metrics arising from the external Kasparov product of spectral triples. We in particular focus on spectral triples for twisted group $C^*$-algebras arising from proper length functions on countable amenable groups.

\subsection*{Acknowledgments}
This research was funded by The Research Council of Norway [project 324944] and [project 345433]. The third author acknowledges the financial support of The Norwegian National Security Authority.

\section{Preliminaries}\label{sec:prelims}

\subsection{Tensor norms, opposite algebras, and traces}\label{sec:tensor-norms-etc}

By a $*$-algebra, we will mean a complex algebra $A$ equipped with an involutive, anti-multiplicative, conjugate-linear map, usually denoted by $a \mapsto a^*$ for $a \in A$. An isomorphism between $*$-algebras will then mean an algebra isomorphism which intertwines the involutions. A norm $\|\cdot\|$ on a $*$-algebra $A$ is called a $C^*$-norm when it is submultiplicative, invariant under the involution and satisfies the $C^*$-equality $\|a^* a\| = \|a\|^2$ for all $a\in A$. The completion of $A$ with respect to a $C^*$-norm on $A$ becomes a $C^*$-algebra after extending the multiplication and the involution by continuity.

If $A$ and $B$ are $C^*$-algebras, then the algebraic tensor product $A \odot B$ of $A$ and $B$ becomes a $*$-algebra with respect to the involution determined by $(a \otimes b)^* = a^* \otimes b^*$ for all $a \in A$ and $b \in B$ (see for instance \cite[p.~188-189]{Murphy1990}). Moreover, if $\delta$ is a $C^*$-norm on $A \odot B$, then we  will denote the completion of $A \odot B$ with respect to $\delta$ by $A \otimes_\delta B$. There are two $C^*$-norms in particular that are natural to consider.

First: The spatial norm $\| \cdot \|_{\rm min}$ on $A \odot B$ is defined for any $x \in A \odot B$ by
\begin{equation*}
    \|x\|_{\rm min} = \|(\pi_A \otimes \pi_B)(x)\|_{B(H_A \otimes H_B)},
\end{equation*}
where $\pi_A \colon A \to B(H_A)$ is any faithful $*$-representation of $A$, and $\pi_B \colon B \to B(H_B)$ is any faithful $*$-representation of $B$. It is well-known that the spatial norm is independent of the choices of $\pi_A$ and $\pi_B$. The spatial norm satisfies $\|x\|_{\rm min} \leq \delta(x)$ for any $x \in A \odot B$ and any $C^*$-norm $\delta$ on $A \odot B$. This is known as Takesaki's Theorem, cf.~\cite[Section 6.4]{Murphy1990} or \cite[Section 3.4]{BO} for a proof. We denote the completion of $A \odot B$ with respect to the spatial norm by $A \otimes_{\rm min} B$ and call this the \emph{minimal tensor product of $A$ and $B$}.

Second: The maximal norm $\| \cdot \|_{\rm max}$ on $A \odot B$ is defined for any $x \in A \odot B$ by
\begin{equation*}
    \|x\|_{\rm max} = \sup\{ \|\pi(x)\|_{B(H)} : \pi \colon A \odot B \to B(H) \text{ is a cyclic $*$-representation} \}.
\end{equation*}
We denote the completion of $A \odot B$ with respect to the maximal norm by $A \otimes_{\rm max} B$ and call this the \emph{maximal tensor product of $A$ and $B$}. The maximal tensor product satisfies the following universal property: If $\pi \colon A \odot B \to C$ is any $*$-homomorphism, where $C$ is a $C^*$-algebra, then there exists a unique $*$-homomorphism $A \otimes_{\rm max} B \to C$ extending $\pi$, cf.~\cite[Proposition 3.3.7]{BO}.

The universality of the maximal tensor product implies the following: If $\delta$ is a $C^*$-norm on $A \odot B$, then there exists a unique $*$-homomorphism $q_\delta \colon A \otimes_{\rm max} B \to A \otimes_\delta B$ making the following diagram commute
\begin{equation*}
    \xymatrixcolsep{50pt}
    \xymatrixrowsep{10pt}
    \xymatrix{
        & A \otimes_{\rm max} B \ar[dd]^-{q_\delta} \\
        A \odot B \ar[ur]^-{i_{\rm max}} \ar[dr]_-{i_\delta} & \\
        & A \otimes_\delta B,
    }
\end{equation*}
where $i_{\rm max}$ and $i_\delta$ are the respective inclusions of $A \odot B$. By continuity of all the $*$-homomorphisms involved, and density of $A \odot B$, we get that $q_\delta$ is surjective for any $C^*$-norm $\delta$. A $C^*$-algebra $A$ is \emph{nuclear} if $q_\delta$ is injective for all $C^*$-algebras $B$ and any $C^*$-norm $\delta$ on $A \odot B$. By Takesaki's Theorem, we can equivalently say that $A$ is nuclear if $A \otimes_{\rm max} B$ is $*$-isomorphic to $A \otimes_{\rm min} B$ for every $C^*$-algebra $B$. In this case we will simply use the notation $A \otimes B$ to refer to the completion of $A \odot B$ in the necessarily unique $C^*$-norm on $A \odot B$.

By a \emph{trace} $\tau$ on a $C^*$-algebra $B$ we will mean that $\tau$ is 
a nonzero tracial positive functional on $B$. When $\tau$ is normalized, hence a state on $B$, we will call $\tau$ a \emph{tracial state}. We make this distinction to ease the exposition, and in the case where we discuss matrix algebras, $M_n(\bbC)$, the most natural trace to consider is the canonical (unnormalized) one.

If $B$ is a $C^*$-algebra, we will denote by $B^{\textup{op}}=\{ b^{\rm op}: b\in B\}$ the opposite $C^*$-algebra of $B$, with product given by $b_1^{\rm op}b_2^{\rm op} = (b_2b_1)^{\rm op}$ for all $b_1, b_2\in B$, and involution given by $(b^{\rm op})^* = (b^*)^{\rm op}$ for all $b \in B$. The $C^*$-norm on $B^{\rm op}$ is directly inherited from that of $B$, i.e. for any $b \in B$ the norm of $B^{\rm op}$ is defined by $\|b^{\rm op}\| = \|b\|$. The multiplication in $B$ induces a linear map $\mu \colon B \odot B^{\rm op} \to B$ defined by
\begin{equation*}
    \mu(b_1 \otimes b_2^{\rm op}) = b_1b_2.
\end{equation*}
In the event that $B$ admits a trace $\tau$, we can further compose $\mu$ with $\tau$ to get a linear map $\mu_\tau \colon B \odot B^{\rm op} \to \bbC$, which is determined by the following equation:
\begin{equation*}
    \mu_\tau(b_1 \otimes b_2^{\rm op}) = \tau(b_1b_2).
\end{equation*}
It is natural to ask for which $C^*$-norms on $B \odot B^{\rm op}$ the map $\mu_\tau$ is continuous. For the maximal norm we have the following result which follows by elementary methods, cf.~\cite[p.~21]{Ozawa04} for an exposition.

\begin{lemma}\label{lemma:continuity of maximally entangled state for maximal norm}
    Let $B$ be a $C^*$-algebra, and $\tau$ be a trace on $B$. Then $\mu_\tau$ is continuous in the maximal norm.
\end{lemma}

For the spatial norm, we have the following theorem due to Kirchberg.

\begin{theorem}[Kirchberg]\label{thm:Kirchberg's Theorem for continuity of maximally entangled state}
    Let $B$ be a unital $C^*$-algebra with a trace $\tau$. Then $\mu_\tau$ is continuous in the spatial norm if and only if $\tau$ is amenable.
\end{theorem}

These are only two of several equivalent statements proved by Kirchberg distributed over the papers \cite{Kir93} and \cite{Kir94}. A proof of these equivalences can be found in \cite[Theorem 6.2.7]{BO}. Note that by amenability of $\tau$ we mean that the tracial state $\frac{1}{\tau(1_B)}\tau$ is amenable in the following sense:

\begin{definition}[\cite{BO}]\label{def:amenability of tau}
    Let $A \subset B(H)$ be a concretely represented unital $C^*$-algebra. A state $\tau$ on $A$ is called an \emph{amenable trace} if there exists a state $\varphi$ on $B(H)$ such that
    \begin{enumerate}
        \item $\varphi|_A = \tau$, and
        \item $\varphi(uTu^*) = \varphi(T)$ for every unitary $u \in A$ and $T \in B(H)$.
    \end{enumerate}
\end{definition}

\begin{remark}
    The state $\varphi$ in \cref{def:amenability of tau} is sometimes referred to as a \emph{hypertrace} or an \emph{$A$-central state}.
\end{remark}

\begin{remark}
    In general, we will denote the unique extension of $\mu_\tau$ to $B \otimes_{\rm max} B^{\rm op}$ by $\mu_\tau^{\rm max}$. Likewise, if $\mu_\tau$ extends to a continuous map on $B \otimes_{\rm min} B^{\rm op}$, we denote the extension by $\mu_\tau^{\rm min}$. In case $B$ is nuclear, no extra decoration will be made to denote the extension of $\mu_\tau$ to $B \otimes B^{\rm op}$.
\end{remark}

Note that in \cite{BO}, $\mu_\tau$ is defined in a slightly different way, which makes it clear that $\mu_\tau$ is positive. We include a proof of this fact for the sake of completeness.

\begin{lemma}\label{lemma:positivity of maximally entangled state}
    Let $B$ be a $C^*$-algebra with a trace $\tau$. Then the linear functional $\mu_\tau^{\rm max}$ is a positive functional on $B \otimes_{\rm max} B^{\rm op}$. Furthermore, if $B$ is unital and $\tau$ is amenable, then $\mu_\tau^{\rm min}$ is a positive functional on $B \otimes_{\rm min} B^{\rm op}$.
\end{lemma}
\begin{proof}
    Let $x = \sum_{j=1}^n b_{1,j} \otimes b_{2,j}^{\rm op} \in B \odot B^{\rm op}$. Then we compute that
    \begin{align*}
        \mu_{\tau}(x^*x)
        & = \sum_{i,j=1}^n \mu_{\tau}(b_{1,j}^*b_{1,i} \otimes (b_{2,j}^*)^{\rm op}b_{2,i}^{\rm op})
        = \sum_{i,j=1}^n \mu_{\tau}(b_{1,j}^*b_{1,i} \otimes (b_{2,i}b_{2,j}^*)^{\rm op}) \\
        & = \sum_{i,j=1}^n \tau(b_{1,j}^*b_{1,i}b_{2,i}b_{2,j}^*)
        = \sum_{i,j=1}^n \tau(b_{2,j}^*b_{1,j}^*b_{1,i}b_{2,i}) \\
        & = \sum_{i,j=1}^n \tau((b_{1,j}b_{2,j})^*b_{1,i}b_{2,i})
        = \tau\left(\sum_{j=1}^n(b_{1,j}b_{2,j})^*\sum_{i=1}^n(b_{1,i}b_{2,i})\right) \\
        & = \tau\left(\left(\sum_{j=1}^n b_{1,j}b_{2,j}\right)^*\left(\sum_{i=1}^n b_{1,i}b_{2,i}\right)\right).
    \end{align*}
    Hence, $\mu_{\tau}$ is positive on $B \odot B^{\rm op}$. Positivity on $B \otimes_{\rm max} B^{\rm op}$ follows from \cref{lemma:continuity of maximally entangled state for maximal norm}. In the case where $B$ is unital and $\tau$ is amenable, positivity of $\mu_\tau$ on $B \otimes_{\rm min} B^{\rm op}$ follows by \cref{thm:Kirchberg's Theorem for continuity of maximally entangled state}, completing the proof.
\end{proof}

We are going to need the following result throughout this paper, cf.~\cite[Theorem 3.5.3]{BO} for a proof of this fact.

\begin{lemma}\label{lemma:Continuity of tensor product maps}
    If $F \colon A \to B$ and $G \colon C \to D$ are completely positive maps between $C^*$-algebras, we have that the map
    \begin{equation*}
        F \odot G \colon A \odot C \to B \odot D
    \end{equation*}
    extends to a completely positive map on both the minimal and maximal tensor products. Denoting (either of) the extensions by $F \otimes G$ we have that
    \begin{equation*}
        \|F \otimes G\| = \|F\|\|G\|.
    \end{equation*}
\end{lemma}

We will often work with multiple tensor products and how these interact with opposite algebras. Here we establish some notation: Let $A$, $B$, and $C$ be $C^*$-algebras. Define $\Sigma \colon A \odot B \to B \odot A$ by
\begin{equation*}
    \Sigma(a \otimes b) = b \otimes a.
\end{equation*}
This is a $*$-isomorphism, hence $\Sigma$ is continuous in every $C^*$-norm on $A \odot B$. We will therefore not use any decoration to indicate the continuous extension of $\Sigma$, and we will furthermore not use any decoration to indicate the domain and codomain of $\Sigma$ either. If there are multiple tensor products present, we will use the subscript notation $\Sigma_{[ij]}$ to denote that we are interchanging the $i$-th and $j$-th tensor factor. For example, $\Sigma_{[23]} \colon A \odot B \odot C \to A \odot C \odot B$ is defined to be $\id_A \otimes \Sigma$.

The map $B \to B^{\rm op}$, $b \mapsto b^{\rm op}$ defines an anti-$*$-isomorphism, which moreover is a $*$-isomorphism if and only if $B$ is commutative. In particular the assignment $b \mapsto b^{\rm op}$ is positive, since
\begin{equation*}
    b^*b \mapsto (b^*b)^{\rm op} = b^{\rm op}(b^*)^{\rm op} = b^{\rm op}(b^{\rm op})^*.
\end{equation*}
Hence the map $\Sigma^{\rm op} \colon A \odot B^{\rm op} \to B \odot A^{\rm op}$ defined by
\begin{equation*}
    \Sigma^{\rm op}(a \otimes b^{\rm op}) = b \otimes a^{\rm op}
\end{equation*}
is a positive linear isomorphism, which furthermore is continuous as a linear map
\begin{align*}
    \Sigma^{\rm op} & \colon (A \odot B^{\rm op}, \| \cdot \|_{\rm min}) \to (B \odot A^{\rm op}, \| \cdot \|_{\rm min}), \text{ and} \\
    \Sigma^{\rm op} & \colon (A \odot B^{\rm op}, \| \cdot \|_{\rm max}) \to (B \odot A^{\rm op}, \| \cdot \|_{\rm max}).
\end{align*}
In fact, if $\delta$ is a $C^*$-norm on $B \odot A^{\rm op}$, then we get an analogous $C^*$-norm, $\delta^{\rm op}$ on $A \odot B^{\rm op}$ defined for any $x \in A \odot B^{\rm op}$ by
\begin{equation*}
    \delta^{\rm op}(x) = \delta(\Sigma^{\rm op}(x)).
\end{equation*}
This indeed defines a $C^*$-norm since $\Sigma^{\rm op}$ is an anti-$*$-isomorphism. Then, by construction, we have that $\Sigma^{\rm op}$ is continuous as a linear map
\begin{equation*}
    \Sigma^{\rm op} \colon (A \odot B^{\rm op}, \delta^{\rm op}) \to (B \odot A^{\rm op}, \delta),
\end{equation*}
for any $C^*$-norm $\delta$.

We will again not decorate $\Sigma^{\rm op}$ to indicate that we are using the continuous extension, or to indicate the domain and codomain of $\Sigma^{\rm op}$. If there are multiple tensor factors we will again use the subscript notation $\Sigma_{[ij]}^{\rm op}$ to indicate that we are interchanging the $i$-th and $j$-th tensor factor and passing to the opposite algebra in the $i$-th and $j$-th tensor factor. For example, $\Sigma_{[23]}^{\rm op} \colon A \odot B \odot C^{\rm op} \to A \odot C \odot B^{\rm op}$ is defined to be $\id_A \otimes \Sigma^{\rm op}$.

We will use the notation $\Hom(A,B)$ to denote the vector space of all linear maps from $A$ to $B$, where we simply regard $A$ and $B$ as complex vector spaces. For the continuous dual of $A$ we will use the notation $A^*$, and if $F \colon A \to B$ is a bounded linear map we will denote the dual map by $F^* \colon B^* \to A^*$. The dual map is defined by setting
\begin{equation*}
    F^*\psi = \psi \circ F,
\end{equation*}
for any $\psi \in B^*$. The state space of $A$ is denoted by $S(A)$.

For a unital $C^*$-algebra $A$, we will denote the unit by $1_A$. In the special case where $A = M_n(\bbC)$, we will for simplification use the notation $1_n$ instead of $1_{M_n(\bbC)}$. Similarly, we will use the notation $\id_n$ for the identity map on $M_n(\bbC)$ and $\Tr_n$ for the (unnormalized) trace on $M_n(\bbC)$. We furthermore use the notation $\left\{e_{i,j}^{(n)}\right\}_{i,j=1}^{n}$ to denote the standard matrix units of $M_n(\bbC)$.

\subsection{Lipschitz seminorms and spectral triples}\label{sec:lipschitz-seminorms-and-cqms}

We remind the reader about basic notions from quantum metric geometry, in particular the definition of a Lipschitz seminorm and of a compact quantum metric space. There are numerous definitions in the literature, e.g., \cite{Rieffel2004, RieffelGHdistanceforQMS, RieffelMatricialBridges, Kaad24}, all of which are related and serve different purposes. In the sequel we will find it beneficial to use the terminology from \cite{KaadKyed2025}.

\begin{definition}\label{def:Lipschitz-seminorm}
    Let $A$ be 
    a unital  $C^*$-algebra. A \emph{Lipschitz seminorm} on $A$ is a seminorm $L \colon A \to [0,\infty]$ satisfying
    \begin{enumerate}
        \item $L$ is densely defined in the sense that the domain ${\rm Dom}(L) := \{a \in A \mid L(a) < \infty \}$ is dense in $A$.
        \item $L(a) = L(a^*)$ for all $a \in A$.
        \item $\bbC \cdot 1_A \subseteq \ker(L) := \{ x \in A : L(x) = 0 \}$.
    \end{enumerate}
\end{definition}

In general, by a seminorm on $A$, we will henceforth always mean a seminorm $L \colon A \to [0, \infty]$ that is densely defined in the sense above. Given such a seminorm $L$ on a unital $C^*$-algebra $A$, we may associate to it the \emph{Monge-Kantorovi\v{c} extended metric} $\mk_L \colon S(A) \times S(A) \to [0,\infty]$ given by
\begin{align}\label{eq:def-mk-metric}
    \mk_L (\phi, \psi) := \sup \{ \vert \phi(a) - \psi(a) \vert : a\in A, L(a) \leq 1 \}
\end{align}
It is natural to wonder when $\mk_L$ is a genuine metric on $S(A)$. A result in this direction is that if $\ker(L) = \bbC 1_A$ and 
\begin{equation*}
    \sup\big\{\inf\{\|a-\lambda 1_A\|: \lambda \in \bbC\} : a\in A, L(a)\leq 1\big\} <\infty,
\end{equation*}
then $S(A)$ has finite diameter with respect to~$\mk_L$ (cf.~\cite[Proposition 1.6]{Rieffel98}). When $L$ is a Lipschitz seminorm on $A$, there has been a lot of interest in deciding when $\mk_L$ metrizes the weak$^*$-topology on $S(A)$, that is, when the topology induced from $\mk_L$ coincides with the weak$^*$-topology. When this happens, $(A, L)$ is called a \emph{compact quantum metric space}, and it holds that $\ker(L) = \bbC 1_A$ and $S(A)$ has finite diameter with respect to~$\mk_L$.

A rich source of examples of Lipschitz seminorms on unital $C^*$-algebras comes from spectral triples, of which we now remind the reader. Since we in \cref{sec:Kasparov-products} will  closely follow the exposition of \cite{Kaad24} on Kasparov products, we use the following definition of unital spectral triples, see \cite{Connes1989, ConnesNCGBook, Kaad24}.

\begin{definition}\label{def:spectral-triple}
    A unital spectral triple $(A,H,\partial)$ consists of a unital $C^*$-algebra $A$, a unital faithful representation $\pi\colon A \to B(H)$ on a separable Hilbert space $H$, and a self-adjoint (typically unbounded) operator $\partial \colon {\rm Dom}(\partial) \subseteq H \to H$ such that
    \begin{enumerate}
        \item[(a)] there exists a dense $*$-subalgebra $\calA \subseteq A$ 
        containing the unit $1_A$ 
        such that for every $a \in \calA$ the commutator
        \begin{align*}
            [\partial, \pi(a)] = \partial \pi(a) - \pi(a) \partial
        \end{align*}
        is densely defined and extends to a bounded operator on $H$;
        \item[(b)] the resolvent $(i +\partial)^{-1}$ is a compact operator.
    \end{enumerate}
    A unital spectral triple is \emph{even} when the Hilbert space $H$ is equipped with a $\bbZ /2\bbZ$-grading operator $\gamma \colon H \to H$ such that $\pi(a) \colon H\to H$ is an even operator for all $a \in \calA$ and $\partial \colon {\rm Dom}(\partial) \to H$ is odd. Otherwise we say the unital spectral triple is \emph{odd}.
\end{definition}

In the sequel we will suppress the $*$-homomorphism $\pi$ from the notation.
To a spectral triple $(A, H, \partial)$ we may naturally associate a Lipschitz seminorm $L_\partial \colon A\to [0,\infty]$ through
\begin{align}\label{eq:def-spec-triple-slip-norm}
    L_\partial (a) = \begin{cases}
        \Vert [\partial, a] \Vert & \text{if $a \in \calA$} \\
        \infty & \text{if $a \in A \setminus \calA$},
    \end{cases}
\end{align}
where we for notational convenience have identified $[\partial, a]$ with its bounded extension. It is easy to verify that $L_\partial$ is a Lipschitz seminorm on $A$: The domain may be identified with $\calA$, which is dense in $A$. The statement $\bbC \cdot 1_A \subseteq \ker(L_\partial)$ follows from the unit commuting with $\partial$, and $*$-invariance follows from $\partial$ being self-adjoint.

\begin{remark}\label{remark:extension-by-infty}
    Constructing a Lipschitz seminorm from a spectral triple highlights an important subtlety: We set the domain ${\rm Dom}(L_\partial)$ to be equal to $\calA$, even if it is possible that 
    $\Vert [\partial , a]\Vert < \infty$ for some $a \in A \setminus \calA$. Despite this, we still set $L_\partial (a) = \infty$ in this case. All this is to say that the domain of $L_\partial$ need not agree with the maximal domain for the derivation $a \mapsto [\partial, a]$.
\end{remark}

We will in the sequel need to lift
seminorms from $C^*$-algebras to their tensor products. Let $A$ and $B$ be unital $C^*$-algebras equipped with seminorms $L_A \colon A \to [0,\infty]$ and $L_B \colon B \to [0,\infty]$, respectively. Denote by ${\rm Dom}(L_A)$ and ${\rm Dom}(L_B)$ the domains of $L_A$ and $L_B$, respectively. Letting $A \otimes B$ denote either the minimal or maximal tensor product of $A$ and $B$, we define the following seminorms on $A \otimes B$
\begin{equation}\label{eq:def-tensor-seminorms}
\begin{split}
    (L_A \otimes 1_B) (z)
        &:= \sup\{ L_A((\id_A \otimes \psi)(z)) \colon \psi \in S(B) \} \\
        (1_A \otimes L_B) (z) &:= \sup\{ L_B((\phi \otimes \id_B)(z)) \colon \phi \in S(A) \}\\
        L_{A\otimes B} &:= L_A \otimes 1_B + 1_A \otimes L_B
    \end{split}
\end{equation}
with ${\rm Dom}(L_A \otimes 1_B) = {\rm Dom}(L_A) \odot B$, ${\rm Dom}(1_A \otimes L_B) = A \odot {\rm Dom}(L_B)$, and ${\rm Dom}(L_A \otimes L_B) = {\rm Dom}(L_A) \odot {\rm Dom}(L_B)$, all these domains being dense in $A \otimes B$.

Assume now that $L_A$ and $L_B$ are Lipschitz seminorms. Using the fact that $L_A$ and $L_B$ are both $*$-invariant, it is not difficult to verify that the three seminorms just defined are $*$-invariant. Moreover, we easily verify that $\bbC \cdot (1_A \otimes 1_B)$ is contained in the kernel of all three seminorms. Thus they are all Lipschitz seminorms. Note however that $\ker (L_A \otimes 1_B)$ is generally larger than $\bbC \cdot (1_A \otimes 1_B)$, as it at the very least contains $\bbC \cdot (1_A \otimes {\rm Dom}(L_B))$. The analogous statement holds for $1_A\otimes L_B$. Moreover, $L_{A \otimes B}(x) = L_{B \otimes A}(\Sigma (x))$ for all $x \in A \otimes B$.

Lastly, we will at times need to pass from a $C^*$-algebra $A$ to its opposite $C^*$-algebra $A^{\rm op}$. When $A$ comes equipped with a seminorm $L_A$ we induce a seminorm $L_{A^{\rm op}}$ on $A^{\rm op}$ defined by
\begin{equation}\label{eq:def-opposite-seminorm}
    L_{A^{\rm op}} (a^{\rm op}) := L_A (a)
\end{equation}
from which we also note that $a^{\rm op} \in {\rm Dom}(L_{A^{\rm op}})\subseteq A^{\rm op}$ if and only if $a \in {\rm Dom}(L_A) \subseteq A$.
Clearly, $L_{A^{\rm op}}$ is a Lipschitz seminorm if and only if $L_A$ is a Lipschitz seminorm.

\subsection{Reduced twisted group C*-algebras and length functions}\label{twisted}

Suppose  $G$ is a discrete group and $\sigma \in Z^2(G, \bbT)$ is a $2$-cocycle. The $*$-algebra $C_c(G, \sigma)$ consists of the complex functions on $G$ with finite support, equipped with convolution and involution defined by
\begin{align*}
    (f_1*_\sigma f_2) (x) = \sum_{y \in G} f_1(y) f_2(y^{-1}x) \sigma(y,y^{-1}x) \quad \text{and} \quad f^*(x) = \overline{\sigma(x,x^{-1})} \overline{f(x^{-1})}
\end{align*}
for $f, f_1, f_2 \in C_c(G, \sigma)$ and $x \in G$. Associated with the pair $(G,\sigma)$ is also the \emph{$\sigma$-twisted left regular representation} $\lambda^\sigma$. This is the $*$-representation $ \lambda^\sigma \colon C_c(G,\sigma) \to B(\ell^2(G))$, where
\begin{align*}
    \lambda^\sigma(f) \xi (x) = \sum_{y \in G} f(y) \xi(y^{-1}x) \sigma(y,y^{-1}x)
\end{align*}
for $f \in C_c (G, \sigma)$, $\xi \in \ell^2(G)$ and $x\in G$. Thus $\lambda^\sigma$ is the extension of the convolution on $C_c(G,\sigma)$ to a bounded operator on $\ell^2(G)$. Taking the operator norm closure of $C_c(G,\sigma)$ under the (faithful) $*$-representation $\lambda^\sigma$ we obtain the $\sigma$-twisted reduced group $C^*$-algebra $C^*_r(G, \sigma)$. Equivalently,  $\Cred(G,\sigma)$ is the $C^*$-subalgebra of $B(\ell^2(G))$ generated by $\{\lambda^\sigma_g :g \in G\}$, where $\lambda^\sigma_g := \lambda^\sigma(\delta_g)$ for each $g\in G$. Here $\delta_g \in \ell^2(G)$ is the function defined by
\begin{equation*}
    \delta_g(x) =
    \begin{cases}
        1 & x = g, \\
        0 & x \neq g.
    \end{cases}
\end{equation*}
Note moreover that there is a faithful canonical tracial state $\tau_\sigma \in S(\Cred(G,\sigma))$ determined by
\begin{equation*}
    \tau_\sigma(\lambda^\sigma(f)) = f(1_G) \quad \text{ for all } f\in C_c(G,\sigma).
\end{equation*}

We may also consider the \emph{$\sigma$-twisted right regular representation} $\rho^\sigma$. This is the  $*$-antihomomorphism $\rho^\sigma \colon C_c(G,\sigma) \to B(\ell^2(G))$, where
\begin{align*}
    \rho^\sigma(f) \xi (x) = \sum_{y \in G} \xi(y) f(y^{-1}x) \sigma(y,y^{-1}x)
\end{align*}
for $f \in C_c (G, \sigma)$, $\xi \in \ell^2(G)$ and $x\in G$. This defines a faithful $*$-representation of the opposite algebra $\Cred(G,\sigma)^{\rm op}$.

Let $P(G)$ denote the positive definite functions on $G$, and set
\begin{equation}
    P_1(G) := \{ \phi \in P(G) : \phi(1_G) = 1 \}.
\end{equation}
It is well-known that if $\phi \in P(G)$ (resp. $\phi \in P_1(G)$), then the corresponding multiplier $ M_\phi \colon \Cred (G, \sigma) \to \Cred(G, \sigma)$, determined by
\begin{equation*}
    M_\phi(\lambda^\sigma_g) = \phi(g)\lambda^\sigma_g \quad \text{ for all } g\in G,
\end{equation*}
is a completely positive map such that $\|M_\phi\| = \phi(1_G)$ (resp.~a unital completely positive map), see for instance \cite[Theorem 2.5.11]{BO} (when $\sigma=1$) or \cite[Lemma 2.6]{BedosConti2016}.

We will want to define spectral triples for $C^*_r(G,\sigma)$, for which we record the following definition.
Note that we adopt the notion of length function used in \cite{Connes1989}, which in principle allows for slightly more flexibility than the one requiring the length function to take the value $0$ only in the unit. 
\begin{definition}[\cite{Connes1989}]
    If $G$ is a group, a \emph{length function} $l \colon G \to [0, \infty)$ is a function satisfying the following three conditions:
    \begin{enumerate}
        \item $l(s) = 0$ if $s = 1_G$,
        \item $l(s^{-1}) = l(s)$ for all $s \in G$, and
        \item $l(st) \leq l(s) + l(t)$ for all $s,t \in G$.
    \end{enumerate}
\end{definition}

Suppose moreover that $G$ is countable and the length function $l$ is \emph{proper}, that is, for every $r \geq 0$, the set $l^{-1}([0,r])$ is a finite set. Such a length function gives rise to an essentially self-adjoint unbounded operator $\partial_l$ on $\ell^2(G)$ given by the linear extension of
\begin{equation*}
\begin{split}
    \partial_l \colon C_c(G) & \to \ell^2(G) \\
    \delta_s & \mapsto l(s) \delta_s,\text{ for }s\in G.
\end{split}
\end{equation*}
As $\partial_l$ is essentially self-adjoint, we may take its self-adjoint closure, which we will also denote by $\partial_l$ to ease notation. 
Note that $\Cred(G,\sigma)$ is faithfully represented on $\ell^2(G)$ through the left regular representation $\lambda^\sigma$ of $C_c(G,\sigma)$. 
A straightforward calculation such as in \cite[p.~614]{Rieffel02} shows that for any $f$ in $C_c(G, \sigma)$,  the commutator $[\partial_l, \lambda^\sigma(f)]$ belongs to the reduced $\sigma$-twisted $C^*$-crossed product $\ell^\infty(G) \rtimes_{{\rm red}, \sigma} G$, being here identified with the $C^*$-subalgebra of $B(\ell^2(G))$ generated by $\ell^\infty(G)$ (acting as multiplication operators on $\ell^2(G)$) and $\{\lambda^\sigma_g :g \in G\}$. 
In particular, $[\partial_l, \lambda^\sigma(f)]$ extends to a bounded operator on $B(\ell^2(G))$  for any $f \in C_c(G,\sigma)$, and we thus obtain a spectral triple $(C^*_r(G, \sigma),\ell^2(G),\partial_l)$. As a result, we may define a seminorm $L_{\partial_l}$ through \eqref{eq:def-spec-triple-slip-norm}. 
Analogously, we obtain a spectral triple $(C^*_r(G, \sigma)^{\rm op},\ell^2(G),\partial_l)$, where $\Cred(G,\sigma)^{\rm op}$ is faithfully represented on $\ell^2(G)$ with the right regular representation of $C_c(G,\sigma)$.

We will later need the following property of multipliers associated with $P(G)$ with respect to~the Lipschitz seminorm  $L_{\partial_l}$ on $\Cred(G,\sigma)$.
\begin{lemma}\label{lemma:L-cont-multipliers}
    Let $G$ be a countable discrete group, $l \colon G \to [0,\infty)$ be a proper length function, and $\sigma \in Z^2(G,\bbT)$. Suppose $\phi \in P(G)$. Then $M_\phi$ extends to a completely positive multiplier $T^\phi$ of the twisted crossed product $ \ell^\infty(G) \rtimes_{\mathrm{red}, \sigma} G$, and we have
    \begin{align*}
        L_{\partial_l}(M_\phi (\lambda^\sigma(f))) \leq \Vert M_\phi \Vert \, L_{\partial_l}(\lambda^\sigma(f))
    \end{align*}
    for all $f \in C_c(G, \sigma)$. In particular, if $\phi \in P_1(G)$, $M_\phi$ is an $L_{\partial_l}$-seminorm contraction.
\end{lemma}
\begin{proof}
    The fact that $M_\phi$ extends to a completely positive map $T^\phi$ on $\ell^\infty(G) \rtimes_{\mathrm{red}, \sigma} G$ satisfying
    \begin{align}\label{eq:action-T-phi}
    T^\phi(\psi\lambda^\sigma_g) = \phi(g)\psi\lambda^\sigma_g \quad \text{ for all } \psi \in \ell^\infty(G)  \text{ and } g\in G,
    \end{align}
    and $\|T^\phi\| = \|M_\phi\| = \phi(1_G)$,
    follows by \cite[Corollary 4.3]{BedosConti2015}. Setting $\psi_g (x):= l(g^{-1}x) - l(x)$ for $g, x \in G$ and noting $\psi_g \in \ell^\infty(G)$ by the triangle inequality for $l$ (see \cite[Lemma 5]{Connes1989}), we get that
    \begin{align*}
        [\partial_l, M_\phi(\lambda^\sigma_g)] = \phi(g) \psi_g  \lambda^{\sigma}_g = T^{\phi} (\psi_g \lambda^\sigma_g)
    \end{align*}
    for all $g \in G$. Since $\lambda^\sigma(f) = \sum_{g\in G} f(g) \lambda^\sigma_g$ for $f \in C_c(G, \sigma)$, we deduce that for every such $f$ we have
    \begin{align*}
        [\partial_l, M_\phi(\lambda^\sigma(f))] = T^\phi([\partial_l, \lambda^\sigma(f)]) .
    \end{align*}
    Hence, for each $f \in C_c(G, \sigma)$ we get
    \begin{align*}
        L_{\partial_l}(M_\phi (\lambda^\sigma(f)))
        & = \Vert [\partial_l, M_\phi(\lambda^\sigma(f))] \Vert \\
        & = \Vert T^\phi ([\partial_l, \lambda^\sigma(f)]) \Vert \\
        & \leq \Vert T^\phi \Vert L_{\partial_l}(\lambda^\sigma(f)) \\
        & = \Vert M_\phi \Vert \, L_{\partial_l}(\lambda^\sigma(f)).
    \end{align*}
    At last, if $\phi \in P_1(G)$, then $\|M_\phi\| = 1$, completing the proof.
\end{proof}

\section{Discussions on inducing metrics on unital completely positive maps}\label{sec:discussion-inducing-metrics}

Throughout this section, $A$ and $B$ will be unital $C^*$-algebras, and we will be considering extended metrics on the set $\UCP(A,B)$ of unital completely positive maps from $A$ to $B$. 

As mentioned in the introduction, the first property an extended metric on $\UCP(A,B)$ should satisfy is \emph{stability}. In both \cite{Gilchrist} and \cite{BorovykVellambi25}, a metric $d$ on $\UCP(A,B)$ is said to be stable if
\begin{equation}\label{eq:Gilchrist-stability}
    d(\id_n \otimes F, \id_n \otimes G) = d(F,G)
\end{equation}
holds for every $n \in \bbN$ and $F,G \in \UCP(A,B)$. Note that \eqref{eq:Gilchrist-stability} does not make sense on the nose, because as it currently reads $d$ is a metric not only on the unital completely positive maps between $A$ and $B$, but also simultaneously on their amplifications. In \cite{Gilchrist} and \cite{BorovykVellambi25} it is  expected that an analogue of $d$ exists on $\UCP(M_n(\bbC) \otimes A, M_n(\bbC) \otimes B)$ for every $n \in \bbN$. For our purposes this need not be the case, and we therefore define stability as a property of a sequence of extended metrics.

\begin{definition}[Stability]\label{def:stability}
    Let $A$ and $B$ be unital $C^*$-algebras, and $\{d_n\}_{n \in \bbN}$ be a sequence of extended metrics where for each $n\geq 1$, $d_n$ is an extended metric on
    \begin{equation*}
        \UCP(M_n(\bbC) \otimes A, M_n(\bbC) \otimes B).
    \end{equation*}
    We say that the sequence $\{d_n\}_{n \in \bbN}$ is \emph{stable} if
    \begin{equation}\label{eqn:stability}
        d_n(\id_n \otimes F, \id_n \otimes G) = d_1(F,G)
    \end{equation}
    for every $n \in \bbN$ and $F,G \in \UCP(A,B)$.
\end{definition}

We next formalize the property of \emph{chaining} from \cite{Gilchrist} mathematically as follows.
\begin{definition}[Chaining]\label{def:chaining}
    Let $A$, $B$ and $C$ be unital $C^*$-algebras. We say that extended metrics $d_1$, $d_2$ and $d_3$, defined on $\UCP(A,B)$, $\UCP(B,C)$ and $\UCP(A,C)$, respectively, satisfy \emph{chaining} if given any $F_1, F_2 \in \UCP(A,B)$, and any $G_1, G_2 \in \UCP(B,C)$, the following inequality is satisfied:
    \begin{equation}\label{eqn:chaining}
        d_3(G_1 \circ F_1, G_2 \circ F_2) \leq d_2(G_1,G_2) + d_1(F_1,F_2).
    \end{equation}
    In the special case where $A = B = C$ and $d_1 = d_2 = d_3 = d$, we say that $d$ satisfies chaining.
\end{definition}

Note that in a lot of circumstances it might be too strict to require that \eqref{eqn:chaining} holds for \emph{all} unital completely positive maps. Indeed, even in the matrix algebra case, some restrictions are imposed on at least one of the maps involved in the analogue of \eqref{eqn:chaining}, see for instance \cite[p.~5]{Gilchrist} and \cite[p.~4]{Bussandri23}.

In \cite{Gilchrist} and \cite{BorovykVellambi25}, methods for inducing metrics on quantum channels between matrix algebras $A$ and $B$ from metrics on $S(A)$ (or on $S(A\otimes B)$) are discussed, using the canonical identification of these state spaces with the corresponding spaces of density matrices. However, instead of starting with a metric on any of these state spaces, we  will start with noncommutative geometric data in the form of a seminorm.

First, if $L$ is a seminorm on $A$, we can use the associated Monge-Kantorovi\v{c} extended metric $\mk_L$ on $S(A)$, cf.~\eqref{eq:def-mk-metric}, to induce an extended metric $d_L$ on $\UCP(A,B)$ as follows.

\begin{definition}\label{def:UCP metric D}
    Let $A$ and $B$ be unital $C^*$-algebras, and $L \colon A \to [0,\infty]$ be any seminorm on $A$. Define the extended metric $d_L \colon \UCP(A,B) \times \UCP(A,B) \to [0,\infty]$ by setting
    \begin{equation}\label{eqn:UCP metric D def}
        d_L(F,G) := \sup\limits_{\psi \in S(B)}\mk_L(F^*\psi,G^*\psi)
        = \sup\limits_{\psi \in S(B)}\mk_L(\psi\circ F, \psi\circ G)
    \end{equation}
        for all $F, G \in \UCP(A,B)$. An extended metric induced in this manner will be called a \emph{pullback-induced extended metric from a seminorm}.
\end{definition}

We will study properties of such extended metrics in \cref{UCP}. Note that if we have a sequence $\{L_n\}_{n \in \bbN}$ of seminorms $L_n \colon M_n(\bbC) \otimes A \to [0, \infty]$, it is generally not true that $\{d_{L_n}\}_{n \in \bbN}$ will be a stable sequence of extended metrics. Even if $L_n = \| \cdot \|$, the $C^*$-norm on $M_n(\bbC) \otimes A$, the sequence $\{d_{L_n}\}_{n \in \bbN}$ can not be expected to be stable (see \cite{Wat05}). A way to obtain stability is to use the completely bounded norm $\|\cdot\|_{\rm cb}$ \cite{Paulsen}, that is, to set
\begin{equation*}\label{eqn:stabilization}
    d_n^{\rm stab}(F,G) := \|F-G\|_{\rm cb} =\sup\limits_{m \in \bbN} \|\id_{m} \otimes (F - G)\|,
\end{equation*}
for all $F, G\in \UCP(M_n(\bbC) \otimes A, M_n(\bbC) \otimes B)$. Then $\{d_n^{\rm stab}\}_{n \in \bbN}$ is stable. This is analogous to how the diamond norm on quantum channels is defined in the finite-dimensional case. Even though this sequence of metrics satisfies stability, the completely bounded metric has the drawback that explicit calculations depend on computing distances over all the matrix amplifications (except when $B$ is a matrix algebra, cf.~\cite{Smith1983}). The same problem will occur with the stabilized version of $d_L$ introduced in \cref{UCP}, which exists under natural assumptions. In fact, the computational complexity of $d_L$ itself is high as its definition  requires taking two suprema.

As discussed in \cite{Gilchrist} and \cite{BorovykVellambi25}, there is also a way of inducing a metric on the unital completely positive maps between matrix algebras by employing the Choi-Jamio\l{}kowski isomorphism. By naively assuming that an embedding $\omega \colon \UCP(A,B) \to S(C)$ exists for some unital $C^*$-algebras $A, B$ and $C$, we could, given a seminorm $L$ on $C$, induce an extended metric on $\UCP(A,B)$ through 
\begin{equation}\label{eqn:Delta metric imprecise}
    \Delta_{L}(F,G) := \mk_L(\omega(F),\omega(G)).
\end{equation}
In \cref{sec:tracial-Cstar-algebras-and-cp-maps}, assuming that $B$ has a faithful trace, and choosing  $C$ to be $A \otimes_{\rm max} B^{\rm op}$, or $A \otimes_{\rm min} B^{\rm op}$ under suitable assumptions, we develop the necessary technical tools to define an embedding $\omega$ to make the definition in \eqref{eqn:Delta metric imprecise} precise, see \cref{def:TC metric Delta}. In fact, this will allow us to characterize precisely the maps from $A$ to $B$ that will correspond to states on $A \otimes_{\rm max} B^{\rm op}$ under this embedding, cf. \cref{thm:maps with states as associated functionals}.

\section{Pullback-induced metrics}\label{UCP}

Let $A$ and $B$ be unital $C^*$-algebras, and $L \colon A \to [0,\infty]$ be a seminorm on $A$. We recall from \cref{def:UCP metric D} that $d_L \colon \UCP(A,B) \times \UCP(A,B) \to [0,\infty]$ is the extended metric given by
\begin{equation*}
    d_L(F_1, F_2) := \sup\limits_{\psi \in S(B)}\mk_L(F_1^*\psi,F_2^*\psi) = \sup\limits_{\psi \in S(B)}\mk_L(\psi\circ F_1, \psi\circ F_2)
\end{equation*}
for all $F_1, F_2 \in \UCP(A,B)$. We will in this section establish sufficient conditions for a stabilized version of this metric to satisfy stability, see \cref{Ln2}, and chaining, see \cref{cor:chaining-for-pullback-metric}.

A general problem with $d_L$ is that it is not clear under which conditions it is possible to define for each $n\in \bbN$ a seminorm $L_n$ on ${\rm UCP}(M_n(\bbC)\otimes A, M_n(\bbC)\otimes B)$ such that $L_1= L$  and $\{d_{L_n}\}_{n\in \bbN}$ is stable. One way out of this problem is to use a stabilizing procedure. Before discussing this process, it is worth pointing out that $d_L$ is always stable in the following sense (employing product states):

\begin{proposition}\label{Ln1}
    Let $L:A\to [0, \infty]$  be a seminorm. For each $n\in \bbN$, let $1_n\otimes L$ be the seminorm on $M_n(\bbC)\otimes A$  with ${\rm Dom}(1_n\otimes L) = M_n(\bbC)\odot {\rm Dom}(L)$ given as in $\eqref{eq:def-tensor-seminorms}$, that is,
    \begin{equation*}
        (1_n\otimes L)(x) = \sup\{ L((\varphi \otimes \id_A)(x)) \colon \varphi \in S(M_n(\bbC)) \}.
    \end{equation*}
    Define $\delta_{1_n\otimes L}: {\rm UCP}(M_n(\bbC)\otimes A, M_n(\bbC)\otimes B) \times \UCP(M_n(\bbC)\otimes A, M_n(\bbC)\otimes B) \to [0, \infty]$ by
    \begin{equation*}
        \delta_{1_n\otimes L}(G_1, G_2) := \sup\big\{\mk_{1_n\otimes L}\big((\varphi \otimes \psi) \circ G_1,(\varphi\otimes \psi)\circ G_2\big) : \varphi \in S(M_n(\bbC)), \psi \in S(B)\big\}.
    \end{equation*}
    Then each $\delta_{1_n\otimes L}$ is an extended metric on ${\rm UCP}(M_n(\bbC)\otimes A, M_n(\bbC)\otimes B)$ satisfying that $\delta_{1_n\otimes L} \leq d_{1_n\otimes L}$ and $\delta_{1_1\otimes L} = d_L$. Moreover, $\{\delta_{1_n\otimes L}\}_{n\in \bbN}$ is stable.
\end{proposition}
\begin{proof}
    Let $n\in \bbN$. It is straightforward to verify that $\delta_{1_n\otimes L}$ is an extended pseudometric on ${\rm UCP}(M_n(\bbC)\otimes A, M_n(\bbC)\otimes B)$ such that $\delta_{1_n\otimes L} \leq d_{1_n\otimes L}$. 
    To show the separation property for $\delta_{1_n\otimes L}$, assume $G_1 \neq G_2$ in ${\rm UCP}(M_n(\bbC)\otimes A, M_n(\bbC)\otimes B)$. 
    By density of ${\rm Dom}(1_n\otimes L)$ in $M_n(\bbC)\otimes A$, we can then find $x \in M_n(\bbC)\otimes  A$ such that $(1_n\otimes L)(x) < \infty$, and $G_1(x) \neq G_2(x)$. 
    As the set of product states separates the points of the minimal tensor product of two unital $C^*$-algebras, 
    there exist $\varphi \in S(M_n(\bbC))$ and $\psi \in S(B)$ such that \[(\varphi\otimes\psi)(G_1(x))\neq (\varphi\otimes\psi)(G_2(x)).\]  
    Replacing $x$ with $\frac{1}{(1_n\otimes L)(x)} x$ when $(1_n\otimes L)(x) \neq 0$, we get that  $(1_n\otimes L)(x) \leq 1$ and $|((\varphi\otimes\psi)\circ (G_1-G_2))(x)| > 0$, hence that $\mk_{1_n\otimes L}((\varphi\otimes\psi)\circ G_1, (\varphi\otimes\psi)\circ G_2) > 0$.
    This implies that $\delta_{1_n\otimes L}(G_1, G_2) > 0$. Thus, $\delta_{1_n\otimes L}$ is an extended metric.

    Next, it is not difficult to show that $\mk_{1_n\otimes L}\big(\varphi\otimes \psi_1, \varphi\otimes \psi_2\big) = \mk_L(\psi_1, \psi_2)$ for every $\varphi \in S(M_n(\bbC))$ and $\psi_1, \psi_2 \in S(A)$. (We will actually prove a more general statement later, cf.~\cref{lemma:sufficient condition for stability}). Hence we get
    \begin{align*}
        \delta_{1_n\otimes L}({\rm id}_n \otimes F_1, {\rm id}_n &\otimes F_2) \\
        & = \sup\{\mk_{1_n\otimes L}\big(\varphi\otimes (\psi \circ F_1), \varphi\otimes (\psi\circ F_2)\big) : \varphi \in S(M_n(\bbC)), \psi \in S(B)\}\\
        & = \sup\{\mk_{L}\big(\psi \circ F_1, \psi\circ F_2\big) :  \psi \in S(B)\}\\
        & = d_L(F_1, F_2)
    \end{align*}
    for all $F_1, F_2\in{\rm UCP}(A, B)$.
\end{proof}

The process of producing a stabilized version of $d_L$ will rely on the following general result.

\begin{proposition}\label{StabD}
    Assume that $d_n$ is an extended metric on ${\rm UCP}(M_n(\bbC)\otimes A, M_n(\bbC)\otimes B)$ for each $n \in \bbN$, and that
    \begin{equation}\label{stab-3}
        d_{n}({\rm id}_{n}\otimes F_1, {\rm id}_{n}\otimes F_2)
        \leq  d_{n+1}({\rm id}_{n+1}\otimes F_1, {\rm id}_{n+1}\otimes F_2)
    \end{equation}
    for all $n\in \bbN$ and $F_1, F_2\in{\rm UCP}(A, B)$. For each $n\in \bbN$, set
    \begin{equation}\label{d_n-stab}
        d^{\, \rm stab}_{n}(G_1, G_2) := \sup_{m\in \bbN} \big\{d_{mn}({\rm id}_{m} \otimes G_1, {\rm id}_m\otimes G_2)\big\}
    \end{equation}
    for all $G_1, G_2 \in {\rm UCP}(M_n(\bbC)\otimes A, M_n(\bbC)\otimes B)$. Then we have that
    \begin{equation*}
        d_1^{\, \rm stab}(F_1, F_2) := \lim_{m\to \infty} d_{m}({\rm id}_{m} \otimes F_1, {\rm id}_m\otimes F_2)
    \end{equation*}
    for all $F_1, F_2\in{\rm UCP}(A,B)$, and $\{d^{\, \rm stab}_{n}\}_{n\in \bbN}$ is a stable sequence of extended metrics.
\end{proposition}
\begin{proof}
    Note that in \eqref{d_n-stab} we implicitly identify $M_m(\bbC)\otimes M_n(\bbC)$ with $M_{mn}(\bbC)$, so that ${\rm id}_{m} \otimes G$ belongs to $\UCP(M_{mn}(\bbC)\otimes A, M_{mn}(\bbC)\otimes B)$ for every $G \in {\rm UCP}(M_n(\bbC)\otimes A, M_n(\bbC)\otimes B)$.

    The fact that $d^{\rm stab}_{n}$ is an extended metric on ${\rm UCP}(M_n(\bbC)\otimes A, M_n(\bbC)\otimes B)$ for every $n\in \bbN$ is easily verified. Let $n\in \bbN$ and $F_1, F_2\in{\rm UCP}(A, B)$. Then the assumption \eqref{stab-3} that  the sequence $\{d_{n}({\rm id}_{n}\otimes F_1, {\rm id}_{n}\otimes F_2)\}_{n\in \bbN}$ is monotonically non-decreasing clearly implies that $d_1^{\, \rm stab}(F_1, F_2) = \lim_{k\to \infty}d_{k}({\rm id}_{k}\otimes F_1, {\rm id}_{k}\otimes F_2)$. Moreover, for $m, n \in \bbN$, \eqref{stab-3} also implies that
    \begin{equation*}
        d_{mn}({\rm id}_{mn}\otimes F_1, {\rm id}_{mn}\otimes F_2) \leq d_{(m+1)n}({\rm id}_{(m+1)n}\otimes F_1, {\rm id}_{(m+1)n}\otimes F_2)
    \end{equation*}
    for every $m\in \bbN$, i.e., the sequence $\{d_{mn}({\rm id}_{mn}\otimes F_1, {\rm id}_{mn}\otimes F_2)\}_{m\in \bbN}$ is monotonically non-decreasing. Thus we get
    \begin{align*}
        d^{\, \rm stab}_n({\rm id}_{n}\otimes F_1, {\rm id}_{n}\otimes F_2)
        & = \sup_{m\in \bbN} \big\{ d_{mn}({\rm id}_{m} \otimes{\rm id}_{n} \otimes F_1, {\rm id}_m \otimes {\rm id}_{n} \otimes F_2)\big\}\\
        & = \lim_{m\to \infty} d_{mn}({\rm id}_{mn} \otimes F_1, {\rm id}_{mn} \otimes F_2)\\
        & = \lim_{k\to \infty} d_{k}({\rm id}_{k} \otimes F_1, {\rm id}_{k} \otimes F_2)\\
        & = d_1^{\, \rm stab}(F_1, F_2). \qedhere
    \end{align*}
\end{proof}

For $n\in \bbN$ let $\iota_n$ denote the canonical embedding of $M_n(\bbC)$ into $M_{n+1}(\bbC)$, determined by
\begin{equation*}
    \iota_n\big(e_{i,j}^{(n)} \big) = e_{i,j}^{(n+1)} \quad \text{for all } 1\leq i, j\leq n.
\end{equation*}
The following definition will be helpful.

\begin{definition}
    Let $L:A\to [0, \infty]$  be a seminorm.  We will say that  $\{L_n\}_{n\in \bbN}$ is an \emph{upward directed sequence of seminorms adapted to $L$} when each $L_n$ is a seminorm on $M_n(\bbC)\otimes A$ such that $L=L_1$ and $L_n = L_{n+1}\circ (\iota_n\otimes \id_A)$ for every $n\in \bbN$.
\end{definition}

\begin{example}\label{upward-ex}
    Let $L \colon A \to [0, \infty]$  be a seminorm. Then there exist several possible choices of upward directed sequences of seminorms adapted to $L$. For instance, as is easily verified, we may let $\{L_n^1\}_{n\in \bbN}$ (resp.~$\{L_n^\infty\}_{n\in \bbN}$) be the sequence given by
    \begin{equation*}
        L_n^1(x) = \sum_{i,j=1}^n L(a_{i,j}) \quad \text{(resp. } L_n^\infty(x) = \sup\{ L(a_{i,j}) : 1\leq i,j \leq n\}\text{)}
    \end{equation*}
    for each $n\in \bbN$ and $x=\sum_{i, j=1}^n  e_{i,j}^{(n)} \otimes a_{i, j}$ with $a_{i, j}\in A$ for all $i, j=1, \ldots, n$. A more exotic choice is the sequence $\{L_n^S\}_{n\in \bbN}$ given by
    \begin{equation*}
        L_n^S(x) = \sup\Big\{\sum_{i,j=1}^n \big| \varphi\big(e_{i,j}^{(n)}\big)\big| \, L(a_{i,j}) : \varphi\in S(M_n(\bbC))\Big\}.
    \end{equation*}
    We leave to the reader to check that $\{L_n^S\}_{n\in \bbN}$ is an upward directed sequence of seminorms adapted to $L$.
\end{example}

The following establishes the main stability result for pullback-induced extended metrics from seminorms. 

\begin{theorem}\label{Ln2}
    Let $L \colon A \to [0, \infty]$ be a seminorm and let $\{L_n\}_{n\in \bbN}$ be any upward directed sequence of seminorms adapted to $L$. For each $n\in \bbN$, let $d_n:= d_{L_n}$ be the extended metric on $\UCP(M_n(\bbC)\otimes A, M_n(\bbC)\otimes B)$ associated with $L_n$. In particular, $d_1=d_L$.

    Then \eqref{stab-3} holds for every $n\in \bbN$ and $F_1, F_2 \in \UCP(A, B)$. Thus, $\{ d_n^{\, \rm stab}\}_{n\in \bbN}$ is stable and we have
    \begin{equation*}
        d_L^{\, \rm stab}(F_1, F_2) = \lim_{m\to \infty} d_{m}(\id_m\otimes F_1, \id_m\otimes F_2)
    \end{equation*}
    for all $F_1, F_2 \in \UCP(A, B)$.
\end{theorem}
\begin{proof}
    To prove that \eqref{stab-3} holds for $n\in \bbN$, let  $F_1, F_2 \in \UCP(A, B)$ and set $F:=F_1-F_2$. For each $\varphi \in S(M_n(\bbC)\otimes B)$, we can use the Hahn-Banach extension theorem for states \cite[Theorem 3.3.8]{Murphy1990} to pick $\widetilde\varphi  \in S(M_{n+1}(\bbC)\otimes B)$ such that $\widetilde\varphi \circ (\iota_n\otimes \id_B) = \varphi$. Using that $\{L_n\}_{n\in \bbN}$ is upward directed, we get
    \begin{align*}
        \mk_{L_n}(&\varphi\circ (\id_n\otimes F_1), \varphi\circ (\id_n\otimes F_2)) \\
        & = \sup\big\{ \big| (\varphi\circ (\id_n\otimes F))(x)\big| : x \in M_n(\bbC)\otimes A, L_n(x)\leq 1\big\}\\
        & = \sup\big\{ \big| (\varphi\circ (\id_n\otimes F))(x)\big| : x \in M_n(\bbC)\otimes A, (L_{n+1}\circ (\iota_n\otimes \id_A))(x)\leq 1\big\}\\
        & =  \sup\big\{ \big| ((\widetilde\varphi\circ (\iota_{n}\otimes \id_B)) ((\id_n\otimes F)(x))\big| : x \in M_n(\bbC)\otimes A, L_{n+1}((\iota_n\otimes \id_A)(x))\leq 1\big\}\\
        & =  \sup\big\{ \big| ((\widetilde\varphi\circ (\id_{n+1}\otimes F)) ((\iota_n\otimes \id_A)(x))\big| : x \in M_n(\bbC)\otimes A, L_{n+1}((\iota_n\otimes \id_A)(x))\leq 1\big\}\\
        & \leq \, \sup\big\{ \big| ((\widetilde\varphi\circ (\id_{n+1}\otimes F)) (y)\big| : y \in M_{n+1}(\bbC)\otimes A, L_{n+1}(y)\leq 1\big\}\\
        & = \mk_{L_{n+1}}(\widetilde\varphi\circ (\id_{n+1}\otimes F_1), \widetilde\varphi\circ (\id_{n+1}\otimes F_2))
    \end{align*}
    for every $\varphi\in S(M_n(\bbC)\otimes B)$. Hence
    \begin{align*}
        d_n(\id_n\otimes F_1, \id_n\otimes F_2) &= \sup\{ \mk_{L_n}(\varphi\circ (\id_n\otimes F_1), \varphi\circ (\id_n\otimes F_2)) : \varphi \in S(M_n(\bbC)\otimes B)\} \\
        & \leq\,\sup\{ \mk_{L_{n+1}}(\widetilde\varphi\circ (\id_{n+1}\otimes F_1), \widetilde\varphi\circ (\id_{n+1}\otimes F_2)) : \varphi \in S(M_n(\bbC)\otimes B)\}\\
        & \leq\,\sup\{ \mk_{L_{n+1}}(\psi \circ (\id_{n+1}\otimes F_1), \psi \circ (\id_{n+1}\otimes F_2)) : \psi \in S(M_{n+1}(\bbC)\otimes B)\}\\
        & = d_{n+1}(\id_{n+1}\otimes F_1, \id_{n+1}\otimes F_2),
    \end{align*}
    as we wanted to show. The conclusion follows then from \cref{StabD}.
\end{proof}

We proceed to show that chaining for the same metrics is satisfied under reasonable assumptions.

\begin{proposition}\label{chaining-ucp}
    Let $A$, $B$ and $C$ be unital $C^*$-algebras, $L$ be a seminorm on $A$ and $K$ be a seminorm on $B$. We denote by $d_L$ and $d'_L$ the extended metrics on $\UCP(A, B)$ and $\UCP(A, C)$ associated with $L$, respectively, and by $d_K$ the one on $\UCP(B, C)$ associated with $K$.

    Let $F, F_1, F_2 \in \UCP(A,B)$, and $G, G_1, G_2 \in \UCP(B,C)$. Then we have
   \begin{equation}\label{chaining-ucp1}
        d'_{L}(G \circ F_1, G \circ F_2) \leq d_L(F_1,F_2).
   \end{equation}
    Further, if $K\circ F \leq L$, then
   \begin{equation}\label{chaining-ucp2}
        d'_{L}(G_1 \circ F, G_2 \circ F) \leq d_K(G_1,G_2).
   \end{equation}
   Finally, if $K\circ F_1 \leq L$ or $K\circ F_2 \leq L$, then
    \begin{equation}\label{chaining-ucp3}
        d'_{L}(G_1 \circ F_1, G_2 \circ F_2) \leq d_L(F_1,F_2) + d_K(G_1,G_2)
    \end{equation}
\end{proposition}
\begin{proof}
    First, we have
    \begin{align*}
        d'_{L}(G \circ F_1, G \circ F_2)& = \sup\{ \mk_L(\varphi \circ G \circ F_1, \varphi \circ G \circ F_2) : \varphi \in S(C)\}\\
        & \leq \,\sup\{ \mk_L(\phi \circ F_1, \phi \circ F_2) : \phi \in S(B)\}\\
        & = d_L(F_1,F_2),
    \end{align*}
    that is, \eqref{chaining-ucp1} holds.

    Next, assume that $K\circ F \leq L$. This implies that if $a\in A$ and $L(a)\leq 1$, then $K(F(a)) \leq 1$. Thus, for any $\varphi \in S(C)$, we get
    \begin{align*}
        \mk_L(\varphi\circ G_1 \circ F, \varphi\circ G_2 \circ F)& = \sup\{ \big|\varphi(G_1(F(a))) - \varphi(G_2(F(a)))\big| : a \in A, L(a)\leq 1\}\\
        & \leq \sup\{ \big|\varphi(G_1(F(a))) - \varphi(G_2(F(a)))\big| : a \in A, K(F(a))\leq 1\}\\
        & \leq \sup\{ \big|\varphi(G_1(b)) - \varphi(G_2(b))\big| : b \in B, K(b)\leq 1\}\\
        & = \mk_K(\varphi\circ G_1, \varphi\circ G_2).
    \end{align*}
    Taking the sup over all  $\varphi \in S(C)$ gives that \eqref{chaining-ucp2} holds.

    Finally, assume that $K\circ F_1 \leq L$ or $K\circ F_2 \leq L$. By symmetry, it suffices to consider the case where $K\circ F_2 \leq L$. Then, using the triangle inequality, \eqref{chaining-ucp1} with $G=G_1$, and \eqref{chaining-ucp2} with $F=F_2$,  we get
    \begin{equation*}
        d'_{L}(G_1 \circ F_1, G_2 \circ F_2) \leq d'_{L}(G_1 \circ F_1, G_1 \circ F_2) + d'_{L}(G_1 \circ F_2, G_2 \circ F_2) \leq d_L(F_1,F_2) + d_K(G_1,G_2),
    \end{equation*}
    showing that \eqref{chaining-ucp3} holds.
\end{proof}

As a consequence, chaining for stabilized extended metrics is satisfied in the following form.

\begin{corollary}\label{cor:chaining-for-pullback-metric}
    Let $A$, $B$ and $C$ be unital $C^*$-algebras, $L$ be a seminorm on $A$ and $K$ be a seminorm on $B$. Let $\{L_n\}_{n\in \bbN}$ be an upward directed sequence of seminorms adapted to $L$ and $\{K_n\}_{n\in \bbN}$ be one adapted to $K$. Denote then by $d_L^{\, \rm stab}$, $d_K^{\, \rm stab}$ and ${d}_L^{\,'\,\rm stab}$ the stabilized extended metrics on $\UCP(A, B)$, $\UCP(B, C)$ and $\UCP(A, C)$ associated with these sequences, cf.~\cref{Ln2}.
    
    Let $F, F_1, F_2 \in \UCP(A,B)$, and $G, G_1, G_2 \in \UCP(B,C)$. Then we have
    \begin{equation}\label{chaining-ucp-1}
        {d}_L^{\,'\,\rm stab}(G \circ F_1, G \circ F_2) \leq d_L^{\, \rm stab}(F_1,F_2).
    \end{equation}
    Further, if $K_n\circ (\id_n \otimes F) \leq L_n$ for every $n\in \bbN$, then
    \begin{equation}\label{chaining-ucp-2}
        {d}_L^{\,'\,\rm stab}(G_1 \circ F, G_2 \circ F) \leq d_K^{\, \rm stab}(G_1,G_2).
    \end{equation}
    Finally, if $K_n\circ (\id_n \otimes F_1) \leq L_n$ or $K_n\circ (\id_n \otimes F_2) \leq L_n$ for every $n\in \bbN$, then
    \begin{equation}\label{chaining-ucp-3}
        {d}_L^{\,'\,\rm stab}(G_1 \circ F_1, G_2 \circ F_2) \leq d_L^{\, \rm stab}(F_1,F_2) + d_K^{\, \rm stab}(G_1,G_2)
    \end{equation}
\end{corollary}
\begin{proof}
    These inequalities follow by applying \cref{chaining-ucp} at each level and using the limit formulas for $d_L^{\, \rm stab}$, $d_K^{\, \rm stab}$ and ${d}_L^{\,'\,\rm stab}$ provided by \cref{Ln2}.
\end{proof}

\begin{remark}
    Assume that the seminorms $L$ and $K$ in the theorem above satisfy that $K\circ F\leq L$ for some $F\in \UCP(A, B)$, and that $L_n:=L^\star_n$ and $K_n:=K^\star_n$ for every $n\in \bbN$, where $\star\in \{1, \infty, S\}$, cf.~\cref{upward-ex}. Then it is a routine matter to check that $K_n\circ (\id_n \otimes F) \leq L_n$ holds for every $n\in \bbN$.
\end{remark}

Before concluding this section with a couple of examples, let us point out that if $(A, L)$ is a compact quantum metric space, then $d_L$ is a bounded metric on $\UCP(A, B)$, that is,  $\UCP(A,B)$ has a finite diameter with respect to~$d_L$, the reason being that $S(A)$ is then known to have finite diameter with respect to~$\mk_L$. Note that even if $d_L$ is a bounded metric, and the sequence $\{d_n\}_{n \in \bbN}$ is stable with $d_1=d_L$, where $d_n$ is an extended metric on $\UCP(M_n(\bbC) \otimes A, M_n(\bbC) \otimes B)$, there is a priori no reason to expect $d_n$ (or even $\delta_{1_n \otimes L}$ from \cref{Ln1}) to be a bounded metric on $\UCP(M_n \otimes A, M_n \otimes B)$ for $n\geq 2$. In general, it would be good to know conditions under which the extended metric $d_L$ associated with a seminorm $L$ on $A$ (resp.~$d_L^{\, \rm stab}$ constructed in \cref{Ln2}) is a genuine metric on $\UCP(A, B)$.

\begin{example}\label{ergodic}
    Assume that $A$ admits an ergodic (strongly continuous) action $\alpha$ of a compact group $K$, and that $\kappa$ is a continuous length function on $K$ satisfying that $\kappa(k) = 0$ only if $k = 1_K$.  As shown by Rieffel in \cite[Theorem 2.3]{Rieffel98}, $(A, L_{\alpha, \kappa})$ is then a compact quantum metric space, where $L_{\alpha, \kappa}$ is the Lipschitz seminorm on $A$ defined by
    \begin{equation*}
        L_{\alpha, \kappa}(a) = \sup\big\{{\|\alpha_k(a) -a\|}/{\kappa(k)} : k \in K\setminus\{1_K\}\big\}.
    \end{equation*}
    Setting
    \begin{equation*}
        A_{\alpha, \kappa} := \{ a \in A : L_{\alpha, \kappa}(a) \leq 1 \} = \{ a \in A : \|\alpha_k(a) - a\| \leq \kappa(k) \text{ for all } k \in K \},
    \end{equation*}
    we get a bounded metric $\mk_L$ on $S(A)$, given by
    \begin{equation*}
        \mk_L(\varphi, \psi) = \sup\{ |\varphi(a) -\psi(a)| : a \in A_{\alpha, \kappa}\},
    \end{equation*}
    and an associated bounded metric $d_{L_{\alpha, \kappa}}$ on $\UCP(A, A)$. The sequence $\{\delta_{1_n\otimes L_{\alpha, \kappa}}\}_{n\in \bbN}$ is stable by \cref{Ln1}. Concerning the chaining property from \cref{chaining-ucp} (with $A=B=C$), it will apply whenever $F_1$ or  $F_2$ in $\UCP(A, A)$ is an $L_{\alpha, \kappa}$-contraction. As an example of such a map, let us check that this property holds for any $F\in \UCP(A, A)$ such that $F\circ \alpha_k = \alpha_k \circ F$ for all $k\in K$. Indeed, for every $k\in K$ and $a \in A$, we then have
    \begin{equation*}
        \|\alpha_k(F(a)) -F(a)\| = \|F(\alpha_k(a) -a)\| \leq \|F\| \, \|\alpha_k(a) -a\| =  \|\alpha_k(a) -a\|,
    \end{equation*}
    since $\|F\| = \|F(1_A)\| = 1$. Thus, for every $a\in A$, we get
    \begin{equation*}
        L_{\alpha, \kappa}(F(a)) = \sup_{k \neq 1_K} \big\{{\|\alpha_k(F(a)) -F(a)\|}/{\kappa(k)}\big\} \leq  \sup_{k \neq 1_K} \big\{{\|\alpha_k(a) -a\|}/{\kappa(k)}\big\} = L_{\alpha, \kappa}(a),
    \end{equation*}
    as we wanted to verify.

    To be more concrete, let us assume that the compact group $K$ is abelian. Let then $G=\widehat{K}$ be the dual group of $K$, pick $\sigma \in Z^2(G, \bbT)$, and let $\alpha$ be the ergodic dual action of $K$ on $A:=C_r^*(G, \sigma)$ \cite{Rieffel2004}, determined for each $k\in K$ by $\alpha_k(\lambda_g^\sigma) = \langle k, g\rangle \, \lambda_g^\sigma$ for all $g \in G$. Note that each $\alpha_k$ is the multiplier of $A$ associated with the positive definite function $g\mapsto \langle k, g\rangle$ on $G$. Hence, for each $\phi\in P_1(G)$, the associated multiplier $M_\phi \in \UCP(A, A)$  commutes with every $\alpha_k$, and is therefore an $L_{\alpha, \kappa}$-contraction, regardless of the choice of $\kappa$. Also, by Bochner's Theorem, every $\phi \in P_1(G)$ is the Fourier-Stieltjes transform of a unique probability Borel measure on $K$.
\end{example}

\begin{example}\label{ergodic2}
    As a special case of the previous example, let $\bbZ_n=\{0,1, \ldots, n-1\}$ denote the cyclic group with $n$ elements and set $\zeta := e^{i \frac{2\pi}{n}}$. Consider the ergodic action $\alpha$ of $K=\bbZ_n\times \bbZ_n$ on $A:=B(\ell^2(\bbZ_n))\simeq M_n(\bbC)$ implemented by the discrete Weyl operators \cite[Subsection 4.1.2]{W}, i.e.,
    \begin{equation*}
        \alpha_{(a,b)}(X) = W_{a,b}XW_{a,b}^* \quad \text{for all } X \in A
    \end{equation*}
    for each $(a,b) \in K$, where $W_{a, b}= U^aV^b$, $U$ and $V$ being the unitary operators on $\ell^2(\bbZ_n)$ satisfying $U(e_j)= e_{j+1}$ and $V(e_j) = \zeta^j e_j$ for $j=0, 1, \ldots, n-1$, and $\{e_j\}_{j=0}^{n-1}$ denoting the standard basis of $\ell^2(\bbZ_n)$. The simplest choice of a length function on $K$ is given by $\kappa(a,b)=1$ if $(a,b)\neq (0,0)$, while $\kappa(0,0)=0$. The associated Lipschitz seminorm $L_{\alpha, \kappa}$ on $A$ is then given by
    \begin{equation*}
        L_{\alpha, \kappa}(X) = \max\{\|W_{a, b} X - XW_{a,b}\|: (a,b)\in K\}.
    \end{equation*}
    Linear maps from $A$ into itself that commute with all $\alpha_{a,b}$ are called Weyl-covariant in \cite{W}. A Weyl-covariant map $F$ in $\UCP(A,A)$ can be equivalently described as a mixed Weyl-unitary, that is, $F$ is of the form
    \begin{equation*}
        F(X) = \sum_{(a,b)\in K} p(a,b) W_{a,b}XW_{a,b}^*
    \end{equation*}
    for some probability distribution $p:K\to [0,1]$, cf.~\cite[Corollary 4.15]{W}. Thus we get that the chaining property for $d_{L_{\alpha,\kappa}}$ is satisfied whenever $F_1$ or $F_2$ is a mixed Weyl-unitary.

    Finally, note that $K$ is its own dual and that $A$ is $*$-isomorphic to $C_r^*(K, \sigma)$, where $\sigma\in Z^2(G, \bbT)$ is given by $\sigma((a,b),(a',b'))= \zeta^{ba'}$. Thus we could also have reached the above conclusion by making use of the second part of \cref{ergodic}.
\end{example}

\begin{example}
    Let $G$ be a countable discrete group, $l \colon G \to [0,\infty)$ be a proper length function, and $\sigma \in Z^2(G,\bbT)$, so we can form the spectral triple $(A, \ell^2(G), \partial_l)$, where $A:=C_r^*(G, \sigma)$.  Then we get a sequence $\{\delta_{1_n\otimes L_{\partial_l}}\}_{n\in \bbN}$ which is stable, cf.~\cref{Ln1}. If $(A, L_{\partial_l})$ is a compact quantum metric space, then $d_L$ is a bounded metric. For examples where this is known to hold, see, for instance, \cite{Rieffel02, OzawaRieffel, ChristRieffel, LongWu2021, Farsi-et-al}. Moreover, as shown in \cref{lemma:L-cont-multipliers}, if  $\phi \in P_1(G)$, then $M_\phi\in \UCP(A, A)$ is a $L_{\partial_l}$-contraction; thus this class of unital completely positive maps on $A$  can be used when applying the chaining property from \cref{chaining-ucp}. 
\end{example}

\section{Tracial \texorpdfstring{$C^*$}{C*}-algebras and completely positive maps}\label{sec:tracial-Cstar-algebras-and-cp-maps}

In this section we consider a general $C^*$-algebra $A$ and a $C^*$-algebra $B$ having a trace $\tau$. We will define a map $\omega_\tau \colon \Hom(A,B) \to \Hom(A \odot B^{\rm op}, \bbC)$, which is an analogue of the Choi-Jamio\l{}kowski isomorphism when $A$ and $B$ are matrix algebras. Further, we will characterize the linear maps $F \colon A \to B$ satisfying that $\omega_\tau(F)$ extends to a state on $A \otimes_{\rm max} B^{\rm op}$ and call these maps \emph{trace channels}. In the case where $B$ is unital and $\tau$ is amenable, Kirchberg's Theorem, \cref{thm:Kirchberg's Theorem for continuity of maximally entangled state}, will give that $\omega_\tau(F)$ extends to a state on $A \otimes_{\rm min} B^{\rm op}$ as well, cf. \cref{thm:maps with states as associated functionals}.

\subsection{The map \texorpdfstring{$\omega_\tau$}{omega-tau} and trace channels}\label{sec:omega-tau-and-trace-channels}

\begin{definition}\label{def:omega_tau}
    Let $A$ and $B$ be $C^*$-algebras, and suppose that $\tau$ is a trace on $B$. For any linear map $F \colon A \to B$ we define the linear functional $\omega_\tau(F) \colon A \odot B^{\rm op} \to \bbC$ by setting
    \begin{equation}\label{eqn:omega_tau def}
        \omega_\tau(F)(a \otimes b^{\rm op}) = \tau(F(a)b)
    \end{equation}
    for any $a \in A$ and $b \in B$. By the universal property of the algebraic tensor product, this uniquely defines $\omega_\tau(F)$ as a linear functional on $A \odot B^{\rm op}$. It is clear that the assignment $F \mapsto \omega_\tau(F)$ is linear, hence we get a linear map $\omega_\tau \colon \Hom(A,B) \to \Hom(A \odot B^{\rm op}, \bbC)$.
\end{definition}

\begin{remark}\label{rem:omega_tau as composition with mu_tau}
    Note that $\omega_\tau(F)$ is given by the following composition
    \begin{equation*}
        \xymatrixcolsep{50pt}
        \xymatrix{
            A \odot B^{\rm op} \ar[r]^-{F \odot \id_{B^{\rm op}}}
            & B \odot B^{\rm op} \ar[r]^-{\mu_\tau}
            & \bbC,
        }
    \end{equation*}
 (where $\mu_\tau$ is defined in \cref{sec:tensor-norms-etc}). We will often make use of this fact.
\end{remark}

We are furthermore interested in how $\omega_\tau$ interacts with different $C^*$-norms on $A \odot B^{\rm op}$. In particular, if $\delta$ is a $C^*$-norm on $A \odot B^{\rm op}$, we want to know for which linear maps $F \colon A \to B$ do we have that $\omega_\tau(F)$ is continuous with respect to $\delta$. To that end, we make the following definition.

\begin{definition}\label{def:linear maps with continuous associated functional}
    Let $A$ and $B$ be $C^*$-algebras, and suppose that $\tau$ is a trace on $B$. For a $C^*$-norm $\delta$ on $A \odot B^{\rm op}$ define $\calL_\tau^\delta(A,B)$ to be the set of linear maps $F \colon A \to B$ such that $\omega_\tau(F)$ is continuous with respect to $\delta$. In this case we use the notation $\omega_\tau^\delta(F)$ to denote the (unique) extension of $\omega_\tau(F)$ to $A \otimes_\delta B^{\rm op}$. If either $A$ or $B$ is nuclear, no extra decoration will be made to denote the extension of $\omega_\tau(F)$ to $A \otimes B^{\rm op}$.

    In the special cases where $\delta = \| \cdot \|_{\rm min}$ and $\delta = \| \cdot \|_{\rm max}$, we will simply use the notation $\calL_\tau^{\rm min}(A,B)$ and $\calL_\tau^{\rm max}(A,B)$, respectively.
\end{definition}

\begin{remark}
    In this definition we are considering the possibility that $\omega_\tau(F)$ may be continuous in some $C^*$-norm on $A \odot B^{\rm op}$ without assuming anything about the continuity of $\mu_\tau$. If $\mu_\tau$ is continuous (as it always is with respect to the maximal norm), then continuity of $\omega_\tau(F)$ follows by continuity of $F \odot \id_{B^{\rm op}}$.
\end{remark}

\begin{proposition}\label{prop:properties of maps with cont associated functional}
    Let $A$ and $B$ be $C^*$-algebras. Suppose that $\tau$ is a trace on $B$ and that $\delta$ is a $C^*$-norm on $A \odot B^{\rm op}$. Then,
    \begin{enumerate}
        \item $\calL_\tau^\delta(A,B)$ is a subspace of $\Hom(A,B)$,
        \item the assignment $F \mapsto \omega_\tau^\delta(F)$ defines a linear map $\omega_\tau^\delta \colon \calL_\tau^\delta(A,B) \to (A \otimes_\delta B^{\rm op})^*$, and
        \item if $\tau$ is faithful, then the map $\omega_\tau^\delta$ is a linear embedding.
    \end{enumerate}
\end{proposition}
\begin{proof}
    Let $F \in \calL_\tau^\delta(A,B)$. That $\calL_\tau^\delta(A,B)$ is a vector space follows by the fact that $\omega_\tau$ is linear.

    That $\omega_\tau^\delta$ is linear follows by linearity of $\omega_\tau$ and the fact that $\omega_\tau^\delta(F)$ is the \emph{unique} extension of $\omega_\tau(F)$ to $A \otimes_\delta B^{\rm op}$.

    Suppose now that $\tau$ is faithful. We want to show that $\omega_\tau^\delta(F) = 0$ if and only if $F = 0$. Since $\omega_\tau^\delta(F)$ is the unique extension of $\omega_\tau(F)$ it is sufficient to show that $\omega_\tau(F) = 0$ if and only if $F = 0$. Suppose that $\omega_\tau(F) = 0$. Then for any $a \in A$ and $b \in B$ we have that $\tau(F(a)b) = 0$. Since $\tau$ is faithful this means that $F(a) = 0$ for every $a \in A$, hence $F = 0$.
\end{proof}

We will show that the completely positive maps $\CP(A,B)$ from $A$ to $B$ are precisely the maps $F \in \calL_\tau^{\rm max}(A,B)$ for which $\omega_\tau^{\rm max}(F)$ is positive. First we show that if $F \in \CP(A,B)$, then $F \in \calL_\tau^{\rm max}(A,B)$.

\begin{proposition}\label{prop:cp maps have cont associated functional}
    Let $A$ and $B$ be $C^*$-algebras, and suppose that $\tau$ is a trace on $B$. Then $\CP(A,B) \subset \calL_\tau^{\rm max}(A,B)$.

    Furthermore, if $B$ is unital and $\tau$ is amenable, then $\CP(A,B) \subset \calL_\tau^{\rm min}(A,B)$.
\end{proposition}
\begin{proof}
    In the case that $F$ is completely positive, we have by \cref{lemma:Continuity of tensor product maps} and \cref{lemma:continuity of maximally entangled state for maximal norm} that $\omega_\tau(F)$ is continuous in the maximal norm. This proves the first claim.

    In the case that $B$ is unital and $\tau$ is an amenable trace, we get by \cref{lemma:Continuity of tensor product maps} and \cref{thm:Kirchberg's Theorem for continuity of maximally entangled state} that $\omega_\tau(F)$ is continuous in the spatial norm.
\end{proof}

\begin{remark}
    If $B$ is quasidiagonal, then $B$ has at least one amenable tracial state $\tau$, cf.~\cite[Proposition 7.1.16]{BO}, so $\omega_\tau(F)$ is continuous in the spatial norm on $A \odot B^{\rm op}$ for every $F$ in $\CP(A,B)$ in this case. For examples of unital, separable, non-nuclear, quasidiagonal $C^*$-algebras, see~\cite[Remark 7.3]{Pisier2020}.
\end{remark}

\begin{remark}
    Whenever $F \in {\rm Span}(\CP(A,B))$, i.e., $F$ is decomposable, we have that $\omega_\tau(F)$ is continuous in the maximal norm on $A \odot B^{\rm op}$. This follows in the same way as in the proof of \cref{prop:cp maps have cont associated functional}, now using \cite[Proposition 6.11]{GP}, which implies that
    \begin{equation*}
        F \odot \textup{id}_{B^\textup{op}} \colon A \odot B^\textup{op} \to B \odot B^\textup{op}
    \end{equation*}
    extends to a map
    \begin{equation*}
        F \otimes \textup{id}_{B^\textup{op}} \colon A \otimes_{\rm max} B^\textup{op} \to B \otimes_{\rm max} B^\textup{op}
    \end{equation*}
    for every $F \in {\rm Span}(\CP(A,B))$. On the other hand, if we assume that $\tau$ is amenable, then we may use \cite[Proposition 1.11]{GP} to deduce that $\omega_\tau(F)$ is continuous in the spatial norm on $A \odot B^{\rm op}$ for every $F \in \CB(A,B)$, where $\CB(A,B)$ denotes the set of all completely bounded maps from $A$ to $B$.
\end{remark}

\begin{remark}\label{rem:finite rank cp maps}
    A sufficient condition for the continuity of $\omega_\tau(F)$ in the spatial norm on $A \odot B^{\rm op}$ depending only on $F$ is as follows. Assume $F$ is the pointwise limit of a net of finite rank completely positive maps from $A$ to $B$. Then it follows from \cite[Theorem 7.11]{GP} that the map
    \begin{equation*}
        F \odot \textup{id}_{B^\textup{op}} \colon A \odot B^\textup{op} \to B \odot B^\textup{op}
    \end{equation*}
    extends to a continuous map
    \begin{equation*}
        F \otimes \textup{id}_{B^\textup{op}} \colon A \otimes_{\rm min} B^\textup{op} \to B \otimes_{\rm max} B^\textup{op}.
    \end{equation*}
    As we know from \cref{lemma:continuity of maximally entangled state for maximal norm}, $\mu_{\tau}$ is continuous on $B \odot B^{\textup{op}}$ with respect to the maximal norm, so we get that $\omega_\tau(F)$ is continuous with respect to the spatial norm on $A \odot B^{\rm op}$ as it is given by the composition of the following continuous maps:
    \begin{equation*}
        \xymatrixcolsep{50pt}
        \xymatrix{
            A \otimes_{\rm min} B^{\rm op} \ar[r]^-{F \otimes \id_{B^{\rm op}}} & B \otimes_{\rm max} B^{\rm op} \ar[r]^-{\mu_\tau^{\rm max}} & \bbC
        }
    \end{equation*}
    Hence, $F \in \calL_\tau^{\rm min}(A,B)$ in this case.
\end{remark}

To prove the desired characterization of completely positive maps from $A$ to $B$, we need to understand how the map $\omega_\tau$ interacts with tensor products of linear maps.

\begin{lemma}\label{lemma:flip of omega_tau of tensor prods}
    Let $A$ and $C$ be $C^*$-algebras, and $B$ and $D$ be $C^*$-algebras with traces $\tau_B$ and $\tau_D$, respectively. Then for any $F \in \CP(A,B)$ and $G \in \CP(C,D)$ we have that
    \begin{equation}\label{eqn:flip of omega_tau max tensor}
        (\Sigma_{[23]})^*\omega_{\tau_B \otimes \tau_D}^{\rm max}(F \otimes G)
        = \omega_{\tau_B}^{\rm max}(F) \otimes \omega_{\tau_D}^{\rm max}(G).
    \end{equation}

    Furthermore, if $B$ and $D$ are unital, and $\tau_B$ and $\tau_D$ are amenable, then
    \begin{equation}\label{eqn:flip of omega_tau min tensor}
        (\Sigma_{[23]})^*\omega_{\tau_B \otimes \tau_D}^{\rm min}(F \otimes G)
        = \omega_{\tau_B}^{\rm min}(F) \otimes \omega_{\tau_D}^{\rm min}(G).
    \end{equation}
\end{lemma}
\begin{proof}
    By \cref{lemma:Continuity of tensor product maps} and \cref{prop:cp maps have cont associated functional} we have that all the maps in the following diagram are well-defined.
    \begin{equation}\label{eqn:diagram for flip of omega_tau max tensor}
    \xymatrixcolsep{50pt}
    \xymatrixrowsep{30pt}
        \xymatrix{
            \CP(A,B) \times \CP(C,D) \ar[d]_-{\otimes} \ar[r]^-{\omega_{\tau_B}^{\rm max} \times \omega_{\tau_D}^{\rm max}} & (A \otimes_{\rm max} B^{\rm op})^* \times (C \otimes_{\rm max} D^{\rm op})^* \ar[dd]^-{\otimes} \\
            \CP(A \otimes_{\rm max} C, B \otimes_{\rm max} D) \ar[d]_-{\omega_{\tau_B \otimes \tau_D}^{\rm max}} & \\
            (A \otimes_{\rm max} C \otimes_{\rm max} B^{\rm op} \otimes_{\rm max} D^{\rm op})^* \ar[r]_-{(\Sigma_{[23]})^*} & (A \otimes_{\rm max} B^{\rm op} \otimes_{\rm max} C \otimes_{\rm max} D^{\rm op})^*
        }
    \end{equation}
    We have that this diagram commutes if and only if \eqref{eqn:flip of omega_tau max tensor} holds for all $F \in \CP(A,B)$ and $G \in \CP(C,D)$. Hence, we simply want to show that this diagram commutes. For $a \in A$, $b \in B$, $c \in C$, and $d \in D$ we have that
    \begin{align*}
        ((\Sigma_{[23]})^*\omega_{\tau_B \otimes \tau_D}^{\rm max}(F \otimes G))(a \otimes b^{\rm op} \otimes c \otimes d^{\rm op})
        & = \omega_{\tau_B \otimes \tau_D}^{\rm max}(F \otimes G)(a \otimes c \otimes b^{\rm op} \otimes d^{\rm op}) \\
        & = (\tau_B \otimes \tau_D)((F(a) \otimes G(c))(b \otimes d)) \\
        & = (\tau_B \otimes \tau_D)(F(a)b \otimes G(c)d) \\
        & = \tau_B(F(a)b)\tau_D(G(c)d) \\
        & = \omega_{\tau_B}^{\rm max}(F)(a \otimes b^{\rm op})\omega_{\tau_D}^{\rm max}(G)(c \otimes d^{\rm op}) \\
        & = (\omega_{\tau_B}^{\rm max}(F) \otimes \omega_{\tau_D}^{\rm max}(G))(a \otimes b^{\rm op} \otimes c \otimes d^{\rm op}).
    \end{align*}
    By linearity and continuity of the relevant maps, and the fact that $A \odot B^{\rm op} \odot C \odot D^{\rm op}$ is dense in $A \otimes_{\rm max} B^{\rm op} \otimes_{\rm max} C \otimes_{\rm max} D^{\rm op}$, we get that the above diagram commutes.

    Now suppose that $B$ and $D$ are unital, and that $\tau_B$ and $\tau_D$ are amenable.
    By \cite[Proposition 6.3.5]{BO} we get that $\tau_B \otimes \tau_D$ is an amenable trace on $B \otimes_{\rm min} D$. Hence, everything appearing in the diagram presented in \eqref{eqn:diagram for flip of omega_tau max tensor} is still well-defined with min in place of max by \cref{lemma:Continuity of tensor product maps} and \cref{prop:cp maps have cont associated functional}. That \eqref{eqn:flip of omega_tau min tensor} holds then follows by the same computation we did earlier in this proof.
\end{proof}

\begin{remark}
    Since the $*$-isomorphism $\Sigma_{[23]}$ is idempotent, we have that \eqref{eqn:flip of omega_tau max tensor} and \eqref{eqn:flip of omega_tau min tensor} give rise to analogous equations
    \begin{equation}\label{eqn:flip of omega_tau max tensor alt}
        \omega_{\tau_B \otimes \tau_D}^{\rm max}(F \otimes G)
        = (\Sigma_{[23]})^*(\omega_{\tau_B}^{\rm max}(F) \otimes \omega_{\tau_D}^{\rm max}(G)),
    \end{equation}
    and
    \begin{equation}\label{eqn:flip of omega_tau min tensor alt}
        \omega_{\tau_B \otimes \tau_D}^{\rm min}(F \otimes G)
        = (\Sigma_{[23]})^*(\omega_{\tau_B}^{\rm min}(F) \otimes \omega_{\tau_D}^{\rm min}(G)).
    \end{equation}
\end{remark}

\begin{theorem}\label{thm:characterization of completely positive maps}
    Let $A$ and $B$ be $C^*$-algebras, and let $\tau$ be a faithful trace on $B$. For any $F \in \calL_\tau^{\rm max}(A,B)$ we have that $F$ is completely positive if and only if $\omega_\tau^{\rm max}(F)$ is positive.

    Furthermore, if $B$ is unital, $\tau$ is amenable, and $F \in \calL_\tau^{\rm min}(A,B)$, then $F$ is completely positive if and only if $\omega_\tau^{\rm min}(F)$ is positive.
\end{theorem}
\begin{proof}
    Suppose first that $F \in \calL_\tau^{\rm max}(A,B)$ is completely positive. Since $\omega_\tau^{\rm max}(F)$ and $\mu_\tau^{\rm max} \circ (F \otimes \id_{B^{\rm op}})$ agree on the dense subspace $A \odot B^{\rm op}$, we have that $\omega_\tau^{\rm max}(F)$ is given by the following composition:
    \begin{equation*}
        \xymatrixcolsep{50pt}
        \xymatrix{
            A \otimes_{\rm max} B^{\rm op} \ar[r]^-{F \otimes \id_{B^{\rm op}}} & B \otimes_{\rm max} B^{\rm op} \ar[r]^-{\mu_\tau^{\rm max}} & \bbC.
        }
    \end{equation*}
    The first of these maps is completely positive by \cref{lemma:Continuity of tensor product maps}, and the second map is positive by \cref{lemma:positivity of maximally entangled state}. Hence, $\omega_\tau^{\rm max}(F)$ is a composition of positive maps, therefore $\omega_\tau^{\rm max}(F)$ is positive.

    For the converse, assume that $\omega_\tau^{\rm max}(F)$ is positive. Let $n \in \bbN$ and fix $Y \in M_n(\bbC) \otimes_{\rm max} B$ to be positive. It follows that the
    map $\Tr_n \otimes \tau$ is a faithful trace on $M_n(\bbC) \otimes_{\rm max} B$, cf.~\cite[p. 434]{Avitzour1982}, as $M_n(\bbC) \otimes_{\rm max} B \simeq M_n(\bbC) \otimes_{\rm min} B$. Hence, to see that $F$ is $n$-positive it is enough to check that
    \begin{equation*}
        (\Tr_n \otimes \tau)((\id_n \otimes F)(X)Y)
    \end{equation*}
    is positive for every positive element $X \in M_n(\bbC) \otimes_{\rm max} A$ and every positive element $Y \in M_n(\bbC) \otimes_{\rm  max} B$. We compute that
    \begin{align*}
        (\Tr_n \otimes \tau)((\id_n \otimes F)(X)Y)
        & = (\omega_{\Tr_n \otimes \tau}(\id_n \otimes F))(X \otimes Y^{\rm op}) & \text{by \eqref{eqn:omega_tau def}} \\
        & = (\omega_{\Tr_n \otimes \tau}^{\rm max}(\id_n \otimes F))(X \otimes Y^{\rm op}) & \\
        & = ((\Sigma_{[23]})^*(\omega_{\Tr_n}^{\rm max}(\id_n) \otimes \omega_\tau^{\rm max}(F))) (X \otimes Y^{\rm op}) & \text{by \eqref{eqn:flip of omega_tau max tensor alt}} \\
        & = ((\Sigma_{[23]})^*(\mu_{\Tr_n} \otimes \omega_\tau^{\rm max}(F))) (X \otimes Y^{\rm op}) & \\
        & = ((\mu_{\Tr_n} \otimes \omega_\tau^{\rm max}(F)) \circ \Sigma_{[23]}) (X \otimes Y^{\rm op}). &
    \end{align*}
    Note that $X \otimes Y^{\rm op}$ is positive since both $X$ and $Y$ are assumed to be positive. Recall that $\Sigma_{[23]}$ is a $*$-isomorphism, which in particular means that $\Sigma_{[23]}$ is positive. By \cref{lemma:positivity of maximally entangled state} we have that $\mu_{\Tr_n}$ is positive, and we have assumed that $\omega_\tau^{\rm max}(F)$ is positive. Thus, by \cref{lemma:Continuity of tensor product maps} we have that $\mu_{\Tr_n} \otimes \omega_\tau^{\rm max}(F)$ is positive. Since the composition of positive maps is positive, we have that
    \begin{equation*}
        (\Tr_n \otimes \tau)((\id_n \otimes F)(X)Y)
    \end{equation*}
    is positive. This shows that $F$ is $n$-positive. But since $n$ was arbitrary, $F$ is necessarily completely positive.

    In the case where $B$ is unital and $\tau$ is amenable, the argument given above works mutatis mutandis to show that any $F \in \calL_\tau^{\rm min}(A,B)$ is completely positive if and only if $\omega_\tau^{\rm min}(F)$ is positive.
\end{proof}
Note that in \cref{thm:characterization of completely positive maps}, faithfulness of the trace is not necessary for the forward implications. We will however mostly be concerned with faithful traces in this article, so we have chosen to state the theorem in its present form.

By \cref{thm:characterization of completely positive maps} we have that the linear maps from $A$ to $B$ that will correspond to states on $A \otimes_{\rm max} B^{\rm op}$ under the map $\omega_\tau^{\rm max}$ must necessarily be completely positive. To get a precise characterization in \cref{thm:maps with states as associated functionals} we make the following definition.

\begin{definition}\label{def:trace channels}
    Let $A$ be a unital $C^*$-algebra and $B$ be a $C^*$-algebra with a trace $\tau$. Define the set
    \begin{equation*}
        \TC_\tau(A,B) = \{ F \in \CP(A,B) : \tau(F(1_A)) = 1 \}.
    \end{equation*}
    We will call elements of $\TC_\tau(A,B)$ \emph{trace channels (with respect to $\tau$)}.
\end{definition}
We note that if $B$ is unital and $\tau$ is a tracial state, then $\UCP(A,B)$ is contained in $\TC_\tau(A,B)$. Furthermore, the following result can be proved in a straightforward way. 

\begin{lemma}\label{lemma:trace channels is a convex set}
    Let $A$ be a unital $C^*$-algebra and $B$ be a $C^*$-algebra with a trace $\tau$. Then $\TC_\tau(A,B)$ is a closed convex subset of $\CP(A,B)$.
\end{lemma}

We can now state the main result of this section.

\begin{theorem}\label{thm:maps with states as associated functionals}
    Let $A$ be a unital $C^*$-algebra, $B$ be a $C^*$-algebra with a faithful
    trace
    $\tau$, and $F \in \calL_\tau^{\rm max}(A,B)$. Then the following are equivalent,
    \begin{enumerate}
        \item $\omega_\tau^{\rm max}(F) \in S(A \otimes_{\rm max} B^{\rm op})$, and
        \item $F \in \TC_\tau(A,B)$.
    \end{enumerate}
    In particular, $\omega_\tau^{\rm max}$ restricts to an affine embedding of $\TC_\tau(A,B)$ into $S(A \otimes_{\rm max} B^{\rm op})$.

    Furthermore, if $B$ is unital, $\tau$ is amenable, and $F \in \calL_\tau^{\rm min}(A,B)$, then the following are equivalent,
    \begin{enumerate}
        \item $\omega_\tau^{\rm min}(F) \in S(A \otimes_{\rm min} B^{\rm op})$, and
        \item $F \in \TC_\tau(A,B)$.
    \end{enumerate}
    In this case, $\omega_\tau^{\rm min}$ restricts to an affine embedding of $\TC_\tau(A,B)$ into $S(A \otimes_{\rm min} B^{\rm op})$.
\end{theorem}
\begin{proof}
    By \cref{thm:characterization of completely positive maps} we have that $\omega_\tau^{\rm max}(F)$ is positive if and only if $F$ is completely positive. In this case, we need only show that $\omega_\tau^{\rm max}(F)$ is a state if and only if $\tau(F(1_A)) = 1$. Let $(u_\lambda)_{\lambda \in \Lambda}$ be an approximate unit for $B$. Then $(u_\lambda^{\rm op})_{\lambda \in \Lambda}$ is an approximate unit for $B^{\rm op}$. Note that
    \begin{equation*}
        \|\omega_\tau^{\rm max}(F)\|
        = \lim\limits_{\lambda} (\omega_\tau^{\rm max}(F)(1_A \otimes u_\lambda^{\rm op}))
        = \lim\limits_{\lambda} (\tau(F(1_A)u_\lambda))
        = \tau(F(1_A)),
    \end{equation*}
    hence if $F$ is completely positive, we have that $\omega_\tau^{\rm max}(F)$ is a state if and only if $\tau(F(1_A)) = 1$. That $\omega_\tau^{\rm max}$ restricts to an affine embedding of $\TC_\tau(A,B)$ into $S(A \otimes_{\rm max} B^{\rm op})$ now follows by \cref{lemma:trace channels is a convex set}.

    The proof for the case when $B$ is unital and $\tau$ is amenable follows by exactly the same argument.
\end{proof}

\begin{definition}\label{def:adjointability}
    Let $A$ and $B$ be $C^*$-algebras with faithful traces $\tau_A$ and $\tau_B$, respectively. We say that a linear map $F \colon A \to B$ is \emph{$(\tau_A,\tau_B)$-adjointable} if there exists a linear map $F^\sharp \colon B \to A$ such that
    \begin{equation}\label{eqn:adjointability definition}
        \tau_B(F(a)b) = \tau_A(aF^\sharp(b))
    \end{equation}
    for every $a \in A$ and $b \in B$.
\end{definition}

\begin{remark}\label{rem:uniqueness-of-sharp}
    The requirement that $\tau_A$ and $\tau_B$ are faithful ensures that $F^\sharp$ is uniquely determined by $F$ through \eqref{eqn:adjointability definition}. Moreover, $F^\sharp$ is $(\tau_B,\tau_A)$-adjointable. Uniqueness of adjoints ensures that $(F^\sharp)^\sharp = F$. If $C$ is a $C^*$-algebra with a faithful trace $\tau_C$ and $G$ is $(\tau_B,\tau_C)$-adjointable, then $G \circ F$ is $(\tau_A, \tau_C)$-adjointable, and uniqueness of the adjoint also implies that
    \begin{equation}\label{eqn:adjoint of composition}
        (G \circ F)^\sharp = F^\sharp \circ G^\sharp.
    \end{equation}
\end{remark}

\begin{proposition}\label{prop:adjoint of trace channel}
    Let $A$ and $B$ be unital $C^*$-algebras with faithful traces $\tau_A$ and $\tau_B$, respectively. If $F \in \TC_{\tau_B}(A,B)$ is $(\tau_A,\tau_B)$-adjointable, then $F^\sharp \in \TC_{\tau_A}(B,A)$, and we furthermore have that
    \begin{equation}\label{eqn:omega_tau of adjoint}
        \omega_{\tau_A}(F^\sharp) = \omega_{\tau_B}(F) \circ \Sigma^{\rm op}.
    \end{equation}
\end{proposition}
\begin{proof}
    First note that
    \begin{equation*}
        \tau_A(F^\sharp(1_B))
        = \tau_A(1_A F^\sharp(1_B))
        = \tau_B(F(1_A)1_B)
        = \tau_B(F(1_A))
        = 1.
    \end{equation*}
    Hence, it remains to show that $F^\sharp$ is completely positive. For $a \in A$ and $b \in B$ we have that
    \begin{align*}
        \omega_{\tau_A}(F^\sharp)(b \otimes a^{\rm op})
        & = \tau_A(F^\sharp(b)a) \\
        & = \tau_A(aF^\sharp(b)) \\
        & = \tau_B(F(a)b) \\
        & = \omega_{\tau_B}(F)(a \otimes b^{\rm op}) \\
        & = (\omega_{\tau_B}(F) \circ \Sigma^{\rm op})(b \otimes a^{\rm op}).
    \end{align*}
    Hence, the following diagram commutes:
    \begin{equation*}
        \xymatrixcolsep{40pt}
        \xymatrix{
            B \odot A^{\rm op} \ar[dr]_-{\omega_{\tau_A}(F^\sharp)} \ar[rr]^-{\Sigma^{\rm op}} & & A \odot B^{\rm op} \ar[dl]^-{\omega_{\tau_B}(F)} \\
            & \bbC &
        }
    \end{equation*}
    By \cref{prop:cp maps have cont associated functional} we have that $F \in \calL_{\tau_B}^{\rm max}(A,B)$. Hence, both $\Sigma^{\rm op}$ and $\omega_{\tau_B}(F)$ are continuous in the maximal norm, thus $\omega_{\tau_A}(F^\sharp)$ necessarily extends to a continuous map on the maximal tensor product. This shows that $F^\sharp \in \calL_{\tau_A}^{\rm max}(B,A)$.

    By \cref{thm:characterization of completely positive maps} we furthermore have that $\omega_{\tau_B}^{\rm max}(F)$ is positive. Recall that $\Sigma^{\rm op}$ is also positive, hence $\omega_{\tau_A}^{\rm max}(F^\sharp)$ is positive. By \cref{thm:characterization of completely positive maps} we get that $F^\sharp$ is completely positive. This shows that $F^\sharp \in \TC_{\tau_A}(B,A)$.
\end{proof}

\subsection{Relation to the Choi-Jamio\l{}kowski isomorphisms}\label{sec:relation-to-choi-jamiolkowski-isos}

\cref{thm:characterization of completely positive maps} can be used to recover some characterizations of completely positive maps that are known for matrix algebras and finite-dimensional $C^*$-algebras. We discuss some of these relationships
in this subsection.

In the case that a $C^*$-algebra $B$ admits a faithful trace $\tau$, we get a linear embedding of $B$ into its dual space $B^*$ by the assignment $b \mapsto \varphi_b := \tau(b \, \cdot \,)$. When $B$ is finite-dimensional, this is  a linear isomorphism, and the elements of
\begin{equation*}
    D_\tau(B) := \{ b \in B : b \text{ is positive and } \tau(b) = 1\}
\end{equation*}
correspond precisely to the states on $B$ under this isomorphism. In the special case $B = M_n(\bbC)$ and $\tau = \Tr_n$, the canonical trace on $M_n(\bbC)$, elements of $D_\tau(B)$ are called \emph{density matrices}.

\begin{example}\label{example:Choi matrix}
    Let $n, m \in \bbN$, $A = M_n(\bbC)$, $B= M_m(\bbC)$, and $\tau$ be the canonical trace on $B$. Let $F \colon A \to B$ be a linear map. We recall, see for instance \cite[Chapter 4]{ES}, that the \emph{Choi matrix} $C_F$ of $F$ is given by
    \begin{equation*}
        C_F := \sum_{i,j = 1}^n e_{i,j}^{(n)} \otimes F(e_{i,j}^{(n)}) \, \in \, M_n(\bbC) \otimes M_m(\bbC) \simeq M_{nm}(\bbC).
    \end{equation*}
    The assignment $F \mapsto C_F$ establishes an isomorphism between $\Hom(A,B)$ and $A \otimes B$. This isomorphism is called the \emph{Choi-Jamio\l{}kowski isomorphism}.

    Moreover, we also recall that the \emph{dual functional} $\widehat F$ on $A \otimes B = M_n(\bbC) \otimes M_m(\bbC)$ is the linear functional satisfying
    \begin{equation*}
        \widehat F(a\otimes b) = {\rm Tr}_m (F(a) b^t)
    \end{equation*}
    for all $a \in M_n(\bbC)$, $b \in M_m(\bbC)$, and that $\widehat F$ is positive if (and only if) $F$ is completely positive, in which case the positive density matrix corresponding to $\widehat F$ is the transpose of $C_F$, cf.~\cite[Lemma 4.2.3 and Theorem 4.2.7]{ES}. Now, let $t$ denote the $*$-isomorphism from $B^{\rm op} = M_m(\bbC)^{\rm op}$ onto $B = M_m(\bbC)$ given by $t(b^{\rm op}) = b^t$ for all $b \in B$. Then one readily verifies that
    \begin{equation*}
        \omega_\tau(F) = \widehat F \circ ({\rm id}_A \otimes t).
    \end{equation*}
    Thus, if $F$ is completely positive and we identify $A \otimes B^{\rm op}$ with $A \otimes B$ via the $*$-isomorphism $\id_A \otimes t$, we get that the density matrix corresponding to the positive functional $\omega_\tau(F)$ is the positive matrix $(C_F)^t$.
\end{example}

\medskip

\begin{example}\label{example:finite-dimensional Choi matrix}
    Let $B$ be a finite-dimensional $C^*$-algebra, and $\tau$ be a faithful trace on $B$ (the existence of such a trace is guaranteed by the structure theorem for finite-dimensional $C^*$-algebras). We get a Hilbert space $(B, \langle \cdot, \cdot \rangle_{\rm GNS})$ where
    \begin{equation*}
        \langle x, y \rangle_{\rm GNS} = \tau(x^*y).
    \end{equation*}
    For any linear
    map $F \colon B \to B$, one assigns an element $\tau^{\rm GNS}(F) \in B \otimes B^{\rm op}$ by choosing a basis $\{b_i\}_{i = 1}^{\dim B}$ for $B$ that is orthonormal with respect to $\langle \cdot, \cdot \rangle_{\rm GNS}$ and setting
    \begin{equation*}
        \tau^{\rm GNS}(F) = \sum_{i = 1}^{\dim B} F(b_i) \otimes (b_i^*)^{\rm op}.
    \end{equation*}
    By \cite[Proposition 3.7]{Wasilewski2024}, $F$ is completely positive if and only if $\tau^{\rm GNS}(F)$ is positive. The functional $\varphi_{\tau^{\rm GNS}(F)}$ evaluated on a basis element $b_m \otimes b_n^{\rm op}$ of $B \otimes B^{\rm op}$ is then given by
    \begin{align*}
        \varphi_{\tau^{\rm GNS}(F)}(b_m \otimes b_n^{\rm op})
        & = (\tau \otimes \tau^{\rm op})\left(\left(\sum_{i = 1}^{\dim B} F(b_i) \otimes (b_i^*)^{\rm op}\right)(b_m \otimes b_n^{\rm op})\right) \\
        & = \sum_{i = 1}^{\dim B} \tau(F(b_i)b_m)\tau(b_i^*b_n) \\
        & = \tau(F(b_n)b_m) \\
        & = \omega_\tau(F)(b_n \otimes b_m^{\rm op}).
    \end{align*}
    Hence, we get that $\omega_\tau(F)$ and $\tau^{\rm GNS}(F)$ are related via the following equality:
    \begin{equation*}
        \omega_\tau(F) = \varphi_{\tau^{\rm GNS}(F)} \circ \Sigma^{\rm op}.
    \end{equation*}
    In \cite{CourtneyGanesanWasilewski25} the element $\tau^{\rm GNS}(F)$ is denoted by $\Choi(F)$ and called the \emph{Choi matrix of $F$}.
\end{example}

\medskip

\begin{example}\label{example:Characterization of cp maps into matrix algebra}
    Let $A$ be a unital $C^*$-algebra, $m \in \bbN$, $B = M_m(\bbC)$, and $\tau$ be the canonical trace on $B$. For each linear map $F \colon A \to M_m(\bbC)$, let $\widehat F$ be the linear functional on $A \otimes M_m(\bbC) \simeq M_m(A)$ determined by
    \begin{equation*}
        \widehat F ( a \otimes e_{i,j}^{(m)}) = F(a)_{i, j} \quad \text{ for every $a\in A$ and $i,j \in \{1, \ldots , m\}$},
    \end{equation*}
    where $F(a)_{i, j}$ denotes the $(i,j)$-coefficient of the matrix $F(a)$. Then, as shown in \cite[Proposition 1.5.14]{BO}, the map $F$ is completely positive if and only if $\widehat F$ is positive, and the map $F \mapsto \widehat F$ gives a bijective correspondence between $\CP(A, M_m(\bbC))$ and the cone of all positive linear functionals on $A \otimes M_m(\bbC)$. Now, it is easy to verify that $\widehat F (a \otimes b) =  {\rm Tr}_m (F(a) b^t)$ for all $a \in A$ and $b\in M_m(\bbC)$. Thus, in analogy with \cref{example:Choi matrix}, $\widehat F$ might be called the dual functional of $F$ on $A\otimes M_m(\bbC)$. Moreover, setting $t(b^{\rm op}):= b^t$ for all $b \in M_m(\bbC)$, it is clear that we have  $\omega_\tau(F) = \widehat F \circ ({\rm id}_A \otimes t)$.
\end{example}

\begin{remark}
    In a von Neumann algebraic setting, there exist some infinite-dimensional versions of the Choi-Jamio\l{}kowski isomorphism in the literature. See for example \cite{Haap2021} and \cite{HKS2024} and references therein.
\end{remark}

\begin{remark}
    When $A$ and $B$ are von Neumann algebras equipped with faithful normal states $\mu$ and $\nu$, respectively,  Duvenhage and Snyman introduce in  \cite[Definition 2.10]{DuvenhageSnyman2018} a notion of a coupling $\omega$ of $(A, \mu)$ and $(B, \nu)$, which is a certain state on $A\odot B'$, where $B'$ denotes the commutant of $B$. Moreover, they associate to each linear map $E:A\to B$ a linear functional $\omega_E$ on $A\odot B'$ and show in \cite[Proposition 3.9]{DuvenhageSnyman2018} that $\omega_E$ is a coupling of $(A, \mu)$ and $(B, \nu)$ if and only if $E\in \UCP(A, B)$ and $\nu\circ E = \mu$. In the case where $\nu$ is tracial, it would be interesting to know if there is a connection between their result and our version of the Choi-Jamio\l{}kowski map, which we leave for future investigations.
\end{remark}

\section{Embedding-induced metrics}\label{sec:induced-metrics-between-cp-maps}

In light of the results of \cref{sec:tracial-Cstar-algebras-and-cp-maps}, we are able to define the embedding-induced metrics alluded to in \eqref{eqn:Delta metric imprecise}.

\begin{definition}\label{def:TC metric Delta}
    Let $A$ be a unital $C^*$-algebra and $B$ be a $C^*$-algebra with a faithful trace $\tau$. Let $L \colon A \otimes_{\rm max} B^{\rm op} \to [0,\infty]$ be a seminorm on $A \otimes_{\rm max} B^{\rm op}$. Define the extended metric $\Delta_{\tau, L}^{\rm max} \colon \TC_\tau(A,B) \times \TC_\tau(A,B) \to [0, \infty]$ by setting
    \begin{equation}\label{eqn:TC metric Delta max def}
        \Delta_{\tau, L}^{\rm max}(F,G) := \mk_L(\omega_\tau^{\rm max}(F),\omega_\tau^{\rm max}(G)).
    \end{equation}
    In the case where $B$ is unital, $\tau$ is amenable, and $L \colon A \otimes_{\rm min} B^{\rm op} \to [0,\infty]$ is a seminorm on $A \otimes_{\rm min} B^{\rm op}$, we define $\Delta_{\tau, L}^{\rm min} \colon \TC_\tau(A,B) \times \TC_\tau(A,B) \to [0, \infty]$ by setting
    \begin{equation}\label{eqn:TC metric Delta min def}
        \Delta_{\tau, L}^{\rm min}(F,G) := \mk_L(\omega_\tau^{\rm min}(F),\omega_\tau^{\rm min}(G)).
    \end{equation}
    An extended metric induced in this way will be called an \emph{embedding-induced extended metric from a seminorm}.
\end{definition}

We will determine some mild criteria under which we can guarantee that stability and chaining hold for these extended metrics. The main result for stability is \cref{thm:stability 1}, while the main results for chaining are \cref{thm:chaining 1}, and \cref{cor:chaining 2}. Note that, like with the pullback-induced metrics discussed in \cref{UCP}, it is generally not obvious when $\Delta_{\tau, L}^{\rm max}$ is a bounded metric on $\TC_\tau(A,B)$ for $C^*$-algebras $A$ and $B$. However, analogously to the situation in \cref{UCP}, $\Delta_{\tau, L}^{\rm max}$ is a bounded metric when $(A \otimes_{\rm max} B^{\rm op}, L)$ is a compact quantum metric space. Note that a necessary condition for this to be the case is that both $A$ and $B$ are unital. The corresponding statement holds true for $\Delta_{\tau, L}^{\rm min}$ and $(A \otimes_{\rm min} B^{\rm op}, L)$. 

\subsection{Stability}\label{sec:stability}

Since we are inducing metrics on the set of trace channels, we need to know that this set is stable under tensor products. This next result establishes this fact.

\begin{lemma}\label{lemma:tensor or trace channels}
    Let $A$ and $C$ be unital $C^*$-algebras, and $B$ and $D$ be $C^*$-algebras with faithful traces $\tau_B$ and $\tau_D$, respectively. Then for any $F \in \TC_{\tau_B}(A,B)$ and $G \in \TC_{\tau_D}(C,D)$ we have that
    \begin{align*}
        F \otimes G & \in \TC_{\tau_B \otimes \tau_D}(A \otimes_{\rm max} C, B \otimes_{\rm max} D), \text{ and } \\
        F \otimes G & \in \TC_{\tau_B \otimes \tau_D}(A \otimes_{\rm min} C, B \otimes_{\rm min} D).
    \end{align*}
\end{lemma}
\begin{proof}
    By \cref{lemma:Continuity of tensor product maps} we have that $F \odot G$ extends to a completely positive map between both $A \otimes_{\rm max} C$ and $B \otimes_{\rm max} D$, and $A \otimes_{\rm min} C$ and $B \otimes_{\rm min} D$. Hence, it remains only to check that $(\tau_B \otimes \tau_D)((F \otimes G)(1_A \otimes 1_C)) = 1$ in both these cases. Indeed, we have
    \begin{equation*}
        (\tau_B \otimes \tau_D)((F \otimes G)(1_A \otimes 1_C))
        = \tau_B(F(1_A))\tau_D(G(1_C))
        = 1. \qedhere
    \end{equation*}
\end{proof}

The following lemma will be useful in the sequel.

\begin{lemma}\label{lemma:sufficient condition for stability}
    Let $A$ be a $C^*$-algebra, $B$ be a unital $C^*$-algebra, $\delta$ be a $C^*$-norm on $B \odot A$, $L$ be a seminorm on $A$, and $K$ be a seminorm on $B \otimes_{\delta} A$. Assume the following conditions are satisfied:
    \begin{enumerate}
        \item $1_B \otimes L \leq K$, and
        \item $K(1_B \otimes x) \leq 1$ for every $x \in A$ satisfying $L(x) \leq 1$.
    \end{enumerate}
    Then
    \begin{equation}\label{eqn:stability technical lemma}
        \mk_L(\varphi_1, \varphi_2) = \mk_{K}(\psi \otimes \varphi_1, \psi \otimes \varphi_2),
    \end{equation}
    for all $\varphi_1, \varphi_2 \in S(A)$ and $\psi \in S(B)$.
\end{lemma}
\begin{proof}
    We compute that for any $\varphi_1, \varphi_2 \in S(A)$ and $\psi \in S(B)$,
    \begin{align*}
        \mk_{L}(\varphi_1, \varphi_2)
        & = \sup\{ |\varphi_1(x) - \varphi_2(x)| \colon x \in A, \ L(x) \leq 1 \} \\
        & = \sup\{ |(\psi \otimes (\varphi_1 - \varphi_2))(1_B \otimes x)| \colon x \in A, \ L(x) \leq 1 \} \\
        & \leq \sup\{ |(\psi \otimes (\varphi_1 - \varphi_2))(X)| \colon X \in B \otimes_{\delta} A, \ K(X) \leq 1 \} \\
        & = \mk_{K}(\psi \otimes \varphi_1, \psi \otimes \varphi_2),
    \end{align*}
    where the inequality follows by assumption (2). Next we compute that
    \begin{align*}
        \mk_{K}(\psi \otimes \varphi_1, \psi \otimes \varphi_2)
        & = \sup\{ |(\psi \otimes (\varphi_1 - \varphi_2))(X)| \colon X \in B \otimes_{\delta} A, \ K(X) \leq 1 \} \\
        & = \sup\{ |(\varphi_1 - \varphi_2)((\psi \otimes \id_A)(X))| \colon X \in B \otimes_{\delta} A, \ K(X) \leq 1 \} \\
        & \leq \sup\{ |\varphi_1(x) - \varphi_2(x)| \colon x \in A, \ L(x) \leq 1 \} \\
        & = \mk_{L}(\varphi_1, \varphi_2),
    \end{align*}
    where the inequality follows because for any $X \in B \otimes_{\delta} A$ we have that
    \begin{equation*}
        L((\psi \otimes \id_A)(X)) \leq (1_B \otimes L)(X) \leq K(X).
    \end{equation*}
    This completes the proof.
\end{proof}

We are now able to state the main theoretical result concerning stability for the extended metrics defined in \cref{def:TC metric Delta}. Note that to attain a sequence of extended metrics of the form $\{\Delta_{\Tr_n \otimes \tau, L_n}^{\rm max}\}_{n \in \bbN}$ we need to be given a sequence of seminorms $\{L_n\}_{n \in \bbN}$. In \cref{thm:stability 1} we impose  compatibility conditions between these seminorms to be able to conclude that $\{\Delta_{\Tr_n \otimes \tau, L_n}^{\rm max}\}_{n \in \bbN}$ is stable. Due to the general setting in which we work, we naturally have to impose some compatibility between the seminorms. However, we do not know if the constraints we have found are optimal, just that they are satisfied for a natural class of examples 
native to 
noncommutative geometry, cf.~\cref{thm:stability 2}.

Recall that $1_n^{\rm op}$ denotes the identity of $M_n(\bbC)^{\rm op}$.

\begin{theorem}\label{thm:stability 1}
    Let $A$ be a unital $C^*$-algebra, and $B$ be a $C^*$-algebra with a faithful trace $\tau$. For each $n \in \bbN$ suppose that we have a seminorm $L_n$ on $(M_n(\bbC) \otimes A) \otimes_{\rm max} (M_n(\bbC) \otimes B)^{\rm op}$ such that
    \begin{enumerate}
        \item $(1_{M_n(\bbC) \otimes M_n(\bbC)^{\rm op}} \otimes L_1) \circ \Sigma_{[23]} \leq L_n$, and
        \item $L_n(\Sigma_{[23]}(1_n \otimes 1_n^{\rm op} \otimes x)) \leq 1$ for every $x \in A \otimes_{\rm max} B^{\rm op}$ satisfying $L_1(x) \leq 1$.
    \end{enumerate}
    Then the sequence $\{\Delta_{\Tr_n \otimes \tau, L_n}^{\rm max}\}_{n \in \bbN}$ of extended metrics on $\TC_{\Tr_n \otimes \tau} (M_n (\bbC) \otimes A , M_n(\bbC) \otimes B)$ is stable.
\end{theorem}
\begin{proof}
    To set ourselves in a situation where we can use \cref{lemma:sufficient condition for stability} we define for each $n \geq 2$ a seminorm $K_n \colon M_n(\bbC) \otimes M_n(\bbC)^{\rm op} \otimes A \otimes_{\rm max} B^{\rm op} \to [0,\infty]$ by setting $K_n = L_n \circ \Sigma_{[23]}$. By conditions (1) and (2) in the statement of the theorem we get that
    \begin{enumerate}
        \item $1_{M_n(\bbC) \otimes M_n(\bbC)^{\rm op}}\otimes L_1 \leq K_n$, and
        \item $K_n(1_n \otimes 1_n^{\rm op} \otimes x) \leq 1$ for every $x \in A \otimes_{\rm max} B^{\rm op}$ satisfying $L_1(x) \leq 1$.
    \end{enumerate}
    Furthermore, we have that for any $\varphi, \psi \in S(M_n(\bbC) \otimes M_n(\bbC)^{\rm op} \otimes A \otimes_{\rm max} B^{\rm op})$, the following equality holds:
    \begin{equation}\label{eqn:stability theorem eqn}
        \mk_{K_n}(\varphi,\psi) = \mk_{L_n}(\Sigma_{[23]}^*\varphi, \Sigma_{[23]}^*\psi).
    \end{equation}

    Now, let $F,G \in \TC_{\tau}(A,B)$, and $n \in \bbN$. Then
    \begin{align*}
        & \Delta_{\Tr_n \otimes \tau, L_n}^{\rm max}(\id_n \otimes F, \id_n \otimes G) & \\
        & = \mk_{L_n}(\omega_{\Tr_n \otimes \tau}^{\rm max}(\id_n \otimes F), \omega_{\Tr_n \otimes \tau}^{\rm max}(\id_n \otimes G)) & \\
        & = \mk_{L_n}(\Sigma_{[23]}^*(\omega_{\Tr_n}^{\rm max}(\id_n) \otimes \omega_{\tau}^{\rm max}(F)), \Sigma_{[23]}^*(\omega_{\Tr_n}^{\rm max}(\id_n) \otimes \omega_{\tau}^{\rm max}(G))) & \text{by \eqref{eqn:flip of omega_tau max tensor alt},} \\
        & = \mk_{K_n}(\omega_{\Tr_n}^{\rm max}(\id_n) \otimes \omega_{\tau}^{\rm max}(F), \omega_{\Tr_n}^{\rm max}(\id_n) \otimes \omega_{\tau}^{\rm max}(G)) & \text{by \eqref{eqn:stability theorem eqn},} \\
        & = \mk_{L_1}(\omega_{\tau}^{\rm max}(F), \omega_{\tau}^{\rm max}(G)) & \text{by \cref{lemma:sufficient condition for stability},} \\
        & = \Delta_{\tau, L_1}^{\rm max}(F,G), &
    \end{align*}
    which is what we wanted to show.
\end{proof}

\begin{remark}\label{rem:stability for spatial norm}
    If $B$ is a unital $C^*$-algebra with an amenable faithful trace $\tau$, then the above result and proof go through if we replace the maximal tensor product with the minimal tensor product everywhere.
\end{remark}

\subsection{Chaining}\label{sec:chaining}

As mentioned at the beginning of this section, we want to talk about chaining for trace channels, but it is not always the case that the composition of two trace channels is again a trace channel. We can however establish that the composition of two trace channels is again a trace channel in the following cases. If $B$ and $C$ are $C^*$-algebras having faithful traces $\tau_B$ and $\tau_C$, respectively, we set
\begin{equation*}
    \QC(B, C, \tau_B, \tau_C):= \{ F\in \CP(B,C): \tau_B=\tau_C\circ F\},
\end{equation*}
that is $\QC(B, C, \tau_B, \tau_C)$ is the set of completely positive trace-preserving maps. 

\begin{lemma}\label{lemma:composition of trace channels}
    Let $A$, $B$, and $C$ be $C^*$-algebras, where $A$ and $B$ are unital, and $B$ and $C$ have faithful
    traces $\tau_B$ and $\tau_C$, respectively. Then the composition of linear maps restricts to maps
    \begin{equation*}
        \xymatrixcolsep{50pt}
        \xymatrixrowsep{10pt}
        \xymatrix{
            \UCP(A,B) \times \TC_{\tau_C}(B,C) \ar[r] & \TC_{\tau_C}(A,C), \\
            \TC_{\tau_B}(A,B) \times \QC(B,C, \tau_B, \tau_C) \ar[r] & \TC_{\tau_C}(A,C).
        }
    \end{equation*}
\end{lemma}
\begin{proof}
    Suppose first that $F \in \UCP(A,B)$ and that $G \in \TC_{\tau_C}(B,C)$. It is clear that $G \circ F$ is completely positive, and we compute that
    \begin{equation*}
        \tau_C(G(F(1_A)))
        = \tau_C(G(1_B))
        = 1.
    \end{equation*}
    Hence, $G \circ F \in \TC_{\tau_C}(A,C)$.

    Suppose now that $F \in \TC_{\tau_B}(A,B)$ and that $G \in \QC(B,C, \tau_B, \tau_C)$. Then we compute that
    \begin{equation*}
        \tau_C(G(F(1_A)))
        = \tau_B(F(1_A))
        = 1.
    \end{equation*}
    This completes the proof.
\end{proof}

With this result in mind, we say that a pair $(F,G)$ is \emph{$\TC$-composable} if
\begin{align*}
    (F,G) & \in \UCP(A,B) \times \TC_{\tau_C}(B,C), \text{ or } \\
    (F,G) & \in \TC_{\tau_B}(A,B) \times \QC(B,C,  \tau_B, \tau_C).
\end{align*}

\begin{remark}
    Note that $\UCP(A,B)$ is a subset of $\TC_{\tau_B}(A,B)$ if (and only if) $\tau_B$ is a tracial \emph{state}. Also, $\QC(B, C, \tau_B, \tau_C)$ is a subset of $\TC_{\tau_C}(B,C)$ whenever $\tau_B$ is a tracial state. This will only impact the wording of \cref{thm:chaining 1} and \cref{cor:chaining 2}, where all traces will be states.
\end{remark}

Before we can state the main results of this section, we are going to need some technical results. This first result tells us how the map $\omega_\tau$ interacts with compositions, the proof of which we choose to omit as it follows by unwrapping definitions and a short computation.

\begin{lemma}\label{lemma:associated functional of composition}
    Let $A$, $B$, and $C$ be $C^*$-algebras, where $C$ has a faithful trace $\tau$. If $F \colon A \to B$ and $G \colon B \to C$ are linear, then
    \begin{equation*}
        \omega_{\tau}(G \circ F) = \omega_{\tau}(G) \circ (F \odot \id_{C^{\rm op}})
    \end{equation*}
    as functionals on $A \odot C^{\rm op}$.
\end{lemma}

If $L$ is a seminorm on $A \otimes_{\rm max} B^{\rm op}$, then we get a seminorm $L^{\rm op}$ on $B \otimes_{\rm max} A^{\rm op}$ by setting $L^{\rm op} = L \circ \Sigma^{\rm op}$. The following result says that this is the ``correct'' seminorm to consider if we want the distance between the adjoints of two channels to be the same as the distance between the original channels.

\begin{lemma}\label{lemma:distance between adjoints}
    Let $A$ and $B$ be unital $C^*$-algebras with faithful traces $\tau_A$ and $\tau_B$ respectively. Let $L$ be a seminorm on $A \otimes_{\rm max} B^{\rm op}$. Suppose that $F,G \in \TC_{\tau_B}(A,B)$ are two $(\tau_A,\tau_B)$-adjointable maps. Then
    \begin{equation*}
        \Delta_{\tau_B,L}^{\rm max}(F,G) = \Delta_{\tau_A,L^{\rm op}}^{\rm max}(F^\sharp,G^\sharp).
    \end{equation*}
    If we furthermore have that $\tau_A$ and $\tau_B$ are amenable traces, and $L$ is a seminorm on $A \otimes_{\rm min} B^{\rm op}$, then
    \begin{equation*}
        \Delta_{\tau_B,L}^{\rm min}(F,G) = \Delta_{\tau_A,L^{\rm op}}^{\rm min}(F^\sharp,G^\sharp).
    \end{equation*}
\end{lemma}
\begin{proof}
    Recall that $\TC_{\tau_B}(A,B) \subset \calL_{\tau_B}^{\rm max}(A,B)$ by \cref{prop:cp maps have cont associated functional}, and if $F \in \TC_{\tau_B}(A,B)$ is $(\tau_A,\tau_B)$-adjointable, then $F^\sharp \in \TC_{\tau_A}(B,A)$ by \cref{prop:adjoint of trace channel}. Hence, both
    \begin{align*}
        \Delta_{\tau_B,L}^{\rm max}(F,G) & \quad \text{ and }  \quad
        \Delta_{\tau_A,L^{\rm op}}^{\rm max}(F^\sharp,G^\sharp)
    \end{align*}
    are well-defined quantities.

    By \eqref{eqn:omega_tau of adjoint} and continuity of $\omega_{\tau_B}(F)$ and $\omega_{\tau_A}(F^\sharp)$ in the respective maximal norms, we get that
    \begin{equation*}
        \omega_{\tau_B}^{\rm max}(F) = \omega_{\tau_A}^{\rm max}(F^\sharp) \circ \Sigma^{\rm op}
    \end{equation*}
    for all $x \in A \otimes_{\rm max} B^{\rm op}$. The same equality holds with $G$ in place of $F$ as well. We then compute that
    \begin{align*}
        & \Delta_{\tau_B,L}^{\rm max}(F,G) & \\
        & = \mk_{L}(\omega_{\tau_B}^{\rm max}(F),\omega_{\tau_B}^{\rm max}(G)) & \\
        & = \mk_{L}(\omega_{\tau_B}^{\rm max}(F^\sharp)\circ\Sigma^{\rm op},\omega_{\tau_B}^{\rm max}(G^\sharp)\circ\Sigma^{\rm op}) & \\
        & = \mk_{L^{\rm op}}(\omega_{\tau_A}^{\rm max}(F^\sharp),\omega_{\tau_A}^{\rm max}(G^\sharp)) & \\
        & = \Delta_{\tau_A,L^{\rm op}}^{\rm max}(F^\sharp,G^\sharp). &
    \end{align*}
    In the case where $\tau_A$ and $\tau_B$ are amenable traces, the same argument given above works mutatis mutandis.
\end{proof}

\begin{remark}\label{rem:sLip-norms and adjoints}
    It follows from \cref{sec:lipschitz-seminorms-and-cqms} that if $L_A$ and $L_B$ are seminorms on $A$ and $B$, respectively, then we have
    \begin{equation}\label{eqn:swap of tensor seminorms}
        (L_{A \otimes_{\rm max} B^{\rm op}})^{\rm op} = L_{B \otimes_{\rm max} A^{\rm op}}.
    \end{equation}
\end{remark}

The following theorem and \cref{cor:chaining 2} are the main results of this section. Especially \cref{thm:chaining 1} is needed to establish that chaining holds for the class of unital completely positive maps between group $C^*$-algebras that arise from positive definite functions, cf.~\cref{thm:chaining 3}. As mentioned in \cref{sec:discussion-inducing-metrics} it is generally too strict to hope that chaining holds in complete generality. Due to the general setting in which we work, some compatibility conditions between the trace channels and the seminorms necessarily have to be imposed. In \cref{thm:chaining 3} we show that these conditions are satisfied by a well-studied class of completely positive maps on twisted group $C^*$-algebras.

\begin{theorem}\label{thm:chaining 1}
    Let $A$, $B$, and $C$ be unital $C^*$-algebras with faithful tracial states $\tau_A$, $\tau_B$, and $\tau_C$ respectively. Let $L$, $K$, and $M$ be seminorms on $A \otimes_{\rm max} B^{\rm op}$, $B \otimes_{\rm max} C^{\rm op}$, and $A \otimes_{\rm max} C^{\rm op}$ respectively. Suppose that
    \begin{enumerate}
        \item $F_1,F_2 \in \TC_{\tau_B}(A,B)$ are $(\tau_A,\tau_B)$-adjointable,
        \item $G_1,G_2 \in \TC_{\tau_C}(B,C)$ are $(\tau_B,\tau_C)$-adjointable,
        \item $(F_1,G_1)$, $(F_2,G_2)$, and $(F_1,G_2)$ are $\TC$-composable,
        \item $K \circ (F_1 \otimes \id_{C^{\rm op}}) \leq M$, and that
        \item $L^{\rm op} \circ (G_2^\sharp \otimes \id_{A^{\rm op}})  \leq M^{\rm op}$.
    \end{enumerate}
    Then
    \begin{equation*}
        \Delta_{\tau_C,M}^{\rm max}(G_1 \circ F_1, G_2 \circ F_2)
        \leq \Delta_{\tau_C,K}^{\rm max}(G_1, G_2) + \Delta_{\tau_B,L}^{\rm max}(F_1, F_2).
    \end{equation*}
\end{theorem}
\begin{proof}
    Note that by \cref{prop:cp maps have cont associated functional} and \cref{prop:adjoint of trace channel}, all of the quantities we will be considering here are well-defined. By the triangle inequality we have that
    \begin{equation*}
        \Delta_{\tau_C,M}^{\rm max}(G_1 \circ F_1, G_2 \circ F_2)
        \leq \Delta_{\tau_C,M}^{\rm max}(G_1 \circ F_1, G_2 \circ F_1) + \Delta_{\tau_C,M}^{\rm max}(G_2 \circ F_1, G_2 \circ F_2).
    \end{equation*}
    We compute that the first of these terms is bounded as follows:
    \begin{align*}
        & \Delta_{\tau_C,M}^{\rm max}(G_1 \circ F_1, G_2 \circ F_1) & \\
        & = \mk_{M}(\omega_{\tau_C}^{\rm max}(G_1 \circ F_1), \omega_{\tau_C}^{\rm max}(G_2 \circ F_1)) & \\
        & = \mk_{M}((F_1 \otimes \id_{C^{\rm op}})^*\omega_{\tau_C}^{\rm max}(G_1), (F_1 \otimes \id_{C^{\rm op}})^*\omega_{\tau_C}^{\rm max}(G_2)) & \text{by \cref{lemma:associated functional of composition}} \\
        & \leq \mk_{K}(\omega_{\tau_C}^{\rm max}(G_1), \omega_{\tau_C}^{\rm max}(G_2)) & \text{since } K \circ (F_1 \otimes \id_{C^{\rm op}}) \leq M \\
        & = \Delta_{\tau_C,K}^{\rm max}(G_1, G_2). &
    \end{align*}
    The second term is bounded by the following computation:
    \begin{align*}
        & \Delta_{\tau_C,M}^{\rm max}(G_2 \circ F_1, G_2 \circ F_2) & \\
        & = \mk_{M}(\omega_{\tau_C}^{\rm max}(G_2 \circ F_1), \omega_{\tau_C}^{\rm max}(G_2 \circ F_2)) & \\
        & = \mk_{M^{\rm op}}(\omega_{\tau_A}^{\rm max}(F_1^\sharp \circ G_2^\sharp), \omega_{\tau_A}^{\rm max}(F_2^\sharp \circ G_2^\sharp)) & \text{by \cref{lemma:distance between adjoints}} \\
        & = \mk_{M^{\rm op}}((G_2^\sharp \otimes \id_{A^{\rm op}})^*\omega_{\tau_A}^{\rm max}(F_1^\sharp), (G_2^\sharp \otimes \id_{A^{\rm op}})^*\omega_{\tau_A}^{\rm max}(F_2^\sharp)) & \text{by \cref{lemma:associated functional of composition}} \\
        & \leq \mk_{L^{\rm op}}(\omega_{\tau_A}^{\rm max}(F_1^\sharp), \omega_{\tau_A}^{\rm max}(F_2^\sharp)) & \text{since } L^{\rm op} \circ (G_2^\sharp \otimes \id_{A^{\rm op}}) \leq M^{\rm op} \\
        & = \mk_{L}(\omega_{\tau_B}^{\rm max}(F_1), \omega_{\tau_B}^{\rm max}(F_2)) & \text{by \cref{lemma:distance between adjoints}} \\
        & = \Delta_{\tau_B,L}^{\rm max}(F_1, F_2), &
    \end{align*}
    Hence, we get that
    \begin{equation*}
        \Delta_{\tau_C,M}^{\rm max}(G_1 \circ F_1, G_2 \circ F_2)
        \leq \Delta_{\tau_C,K}^{\rm max}(G_1, G_2) + \Delta_{\tau_B,L}^{\rm max}(F_1, F_2). \qedhere
    \end{equation*}
\end{proof}

\begin{remark}\label{rem:Chaining 1 for spatial norm}
    If the tracial states $\tau_A$, $\tau_B$, and $\tau_C$ are amenable, then the above result has a version where the maximal tensor product is interchanged with the minimal tensor product instead. The proof of this is completely analogous to the one given above, so we omit the details here.
\end{remark}

Recall the seminorms $L_A \otimes 1_B$ and $1_A \otimes L_B$ defined in \eqref{eq:def-tensor-seminorms}, which play an important role in the following.

\begin{corollary}\label{cor:chaining 2}
    Let $A$, $B$, and $C$ be unital $C^*$-algebras with faithful tracial states $\tau_A$, $\tau_B$, and $\tau_C$, respectively, and with seminorms $L_A$, $L_B$, and $L_C$, respectively. Suppose that
    \begin{enumerate}
        \item $F_1 \in \UCP(A,B)$, $F_2 \in \TC_{\tau_B}(A,B)$ are $(\tau_A,\tau_B)$-adjointable,
        \item $G_1 \in \TC_{\tau_C}(B,C)$, $G_2 \in \QC(B,C, \tau_B, \tau_C)$ are $(\tau_B,\tau_C)$-adjointable,
        \item $L_B \circ F_1 \leq L_A$, and that
        \item $L_B \circ G_2^\sharp \leq L_C$.
    \end{enumerate}
    Then
    \begin{equation*}
        \Delta_{\tau_C,L_{A \otimes_{\rm max} C^{\rm op}}}^{\rm max}(G_1 \circ F_1, G_2 \circ F_2)
        \leq \Delta_{\tau_C,L_{B \otimes_{\rm max} C^{\rm op}}}^{\rm max}(G_1, G_2) + \Delta_{\tau_B,L_{A \otimes_{\rm max} B^{\rm op}}}^{\rm max}(F_1, F_2).
    \end{equation*}
    If we furthermore have that $\tau_A$, $\tau_B$, and $\tau_C$ are amenable, then
    \begin{equation*}
        \Delta_{\tau_C,L_{A \otimes_{\rm min} C^{\rm op}}}^{\rm min}(G_1 \circ F_1, G_2 \circ F_2)
        \leq \Delta_{\tau_C,L_{B \otimes_{\rm min} C^{\rm op}}}^{\rm min}(G_1, G_2) + \Delta_{\tau_B,L_{A \otimes_{\rm min} B^{\rm op}}}^{\rm min}(F_1, F_2).
    \end{equation*}
\end{corollary}
\begin{proof}
    Note that all the pairs $(F_1,G_1)$, $(F_1,G_2)$, and $(F_2,G_2)$ are $\TC$-composable, hence we are in the setting of \cref{thm:chaining 1}. By \cref{rem:sLip-norms and adjoints} we simply need to show that the inequalities $L_B \circ F_1 \leq L_A$ and $L_B \circ G_2^\sharp \leq L_C$ imply that
    \begin{align*}
        L_{B \otimes_{\rm max} C^{\rm op}} \circ (F_1 \otimes \id_{C^{\rm op}}) & \leq L_{A \otimes_{\rm max} C^{\rm op}}, \text{ and } \\
        L_{B \otimes_{\rm max} A^{\rm op}} \circ (G_2^\sharp \otimes \id_{A^{\rm op}}) & \leq L_{C \otimes_{\rm max} A^{\rm op}}.
    \end{align*}
    Let $X \in A \otimes_{\rm max} C^{\rm op}$. Then,
    \begin{align*}
        & L_{B \otimes_{\rm max} C^{\rm op}}((F_1 \otimes \id_{C^{\rm op}})(X)) \\
        & = \sup\limits_{\psi \in S(C)}L_B((F_1 \otimes \psi)(X)) + \sup\limits_{\varphi \in S(B)}L_C(((\varphi \circ F_1) \otimes \id_{C^{\rm op}})(X)) \\
        & = \sup\limits_{\psi \in S(C)}L_B(F_1((\id_A \otimes \psi)(X))) + \sup\limits_{\varphi \in S(B)}L_C(((F_1^*\varphi) \otimes \id_{C^{\rm op}})(X)) \\
        & \leq \sup\limits_{\psi \in S(C)}L_A((\id_A \otimes \psi)(X)) + \sup\limits_{\phi \in S(A)}L_C((\phi \otimes \id_{C^{\rm op}})(X)) \\
        & = L_{A \otimes_{\rm max} C^{\rm op}}(X),
    \end{align*}
    where $F_1^*\varphi \in S(A)$ for every $\varphi \in S(B)$ since $F_1 \in \UCP(A,B)$. Note that the proof of \cref{prop:adjoint of trace channel} specializes to show that $G_2^\sharp \in \UCP(C,B)$ since $G_2 \in \QC(B,C,\tau_B, \tau_C)$. Hence, the second inequality follows by the same argument. Then, by \cref{rem:sLip-norms and adjoints} and \cref{thm:chaining 1} we have that
    \begin{equation*}
        \Delta_{\tau_C,L_{A \otimes_{\rm max} C^{\rm op}}}^{\rm max}(G_1 \circ F_1, G_2 \circ F_2)
        \leq \Delta_{\tau_C,L_{B \otimes_{\rm max} C^{\rm op}}}^{\rm max}(G_1, G_2) + \Delta_{\tau_B,L_{A \otimes_{\rm max} B^{\rm op}}}^{\rm max}(F_1, F_2).
    \end{equation*}
    In the case where $\tau_A$, $\tau_B$, and $\tau_C$ are amenable, the inequality
    \begin{equation*}
        \Delta_{\tau_C,L_{A \otimes_{\rm min} C^{\rm op}}}^{\rm min}(G_1 \circ F_1, G_2 \circ F_2)
        \leq \Delta_{\tau_C,L_{B \otimes_{\rm min} C^{\rm op}}}^{\rm min}(G_1, G_2) + \Delta_{\tau_B,L_{A \otimes_{\rm min} B^{\rm op}}}^{\rm min}(F_1, F_2)
    \end{equation*}
    follows by \cref{rem:Chaining 1 for spatial norm}.
\end{proof}

\section{Applications}\label{sec:applications}

\subsection{The Kasparov external product of spectral triples}\label{sec:Kasparov-products}

Recall that we defined unital spectral triples in \cref{sec:lipschitz-seminorms-and-cqms}. In this section we will show that using the Kasparov external product from $KK$-theory, we can from any spectral triple generate a sequence of seminorms on the matrix amplifications which naturally satisfy attractive stability properties, as discussed in \cref{sec:stability}. This will be made precise below, see \cref{thm:stability 2}. We proceed to closely follow the exposition given in \cite{Kaad24}. Let $(A,H_A,\partial_A)$ be a unital spectral triple, and suppose that $\calB \subseteq B(H_B)$ is a unital $*$-subalgebra. We may then form the essentially self-adjoint operators
\begin{equation*}
    \begin{split}
        \partial_A \otimes 1_{\calB} &\colon {\rm Dom}(\partial_A) \otimes H_B \to H_A \otimes H_B \\
        1_{\calB} \otimes \partial_A &\colon H_B \otimes {\rm Dom}(\partial_A) \to H_B \otimes H_A
    \end{split}
\end{equation*}
where the domains are given by the algebraic tensor product and the codomains are the Hilbert space tensor products of the given Hilbert spaces.
Taking closures, we obtain closed, self-adjoint operators which we by $\partial_A \widehat{\otimes} 1_{\calB}$ and $1_{\calB} \widehat{\otimes} \partial_A$, respectively. 
While $\partial_A \widehat{\otimes} 1_{\calB}$ and $1_{\calB} \widehat{\otimes} \partial_B$ do not generally have compact resolvents as self-adjoint unbounded operators on $H_A \otimes H_B$ and $H_B \otimes H_A$, respectively, we can still use them to induce metrics much in the same way as in \eqref{eq:def-spec-triple-slip-norm}. More precisely, by identifying elements of the algebraic tensor products with operators on the Hilbert space tensor products in the obvious way, we could set
\begin{equation}\label{eq:seminorms-from-partial-diracs}
    \begin{split}
        L_{\partial_A \widehat{\otimes}1_{\calB}} (z) &= \Vert [\partial_A \widehat{\otimes}1_{\calB}, z] \Vert_{B(H_A \otimes H_B)} \quad \text{ for $z \in \calA \odot \calB$} \\
        L_{1_{\calB} \widehat{\otimes}\partial_A}(z) &= \Vert [1_{\calB}\widehat{\otimes}\partial_A, z] \Vert_{B(H_B \otimes H_A)} \quad \text{ for $z \in \calB \odot \calA$} 
    \end{split}
\end{equation}
and extend by $\infty$ in accordance with \cref{remark:extension-by-infty}. Denote by $B$ the $C^*$-closure of $\calB$ in $B(H_B)$.

We proceed to discuss the external Kasparov product of unital spectral triples. Let therefore
$(A,H_A,\partial_A)$ and $(B,H_B,\partial_B)$ be unital spectral triples. We may form the external Kasparov product $(A \otimes_{\rm min} B, H, \partial_A \times \partial_B)$, where the Hilbert space $H$, and the self-adjoint operator $\partial_A \times \partial_B$ depend on the parity of the two spectral triples as follows:

\medskip \textbf{Even times even:} Set $H = H_A \otimes H_B$ and endow $H$ with a $\bbZ/2\bbZ$ grading defined by
\begin{equation*}
    \gamma = \gamma_A \otimes \gamma_B
\end{equation*}
and  $\partial_A \times \partial_B$ to be the closure of the symmetric self-adjoint operator
\begin{equation*}
     \partial_A \otimes \id_{H_B} + \gamma_A \otimes \partial_B \colon {\rm Dom}(\partial_A) \otimes {\rm Dom}(\partial_B) \to H_A \otimes H_B
\end{equation*}

\textbf{Odd times odd:} Set $H = (H_A \otimes H_B) \oplus (H_A \otimes H_B)$ and endow $H$ with a $\bbZ/2\bbZ$ grading defined by
\begin{equation*}
    \gamma =
    \begin{pmatrix}
        1 & 0 \\
        0 & -1
    \end{pmatrix}.
\end{equation*}
The minimal tensor product $A \otimes_{\rm min} B$ is faithfully represented on $H$ via the diagonal representation
\begin{equation*}
    x \mapsto
    \begin{pmatrix}
        x & 0 \\
        0 & x
    \end{pmatrix}.
\end{equation*}
Finally, consider the operator
\begin{equation*}
\begin{split}
    \begin{pmatrix}
        0 & \partial_A \otimes \id_{H_B} + i\id_{H_A} \otimes \partial_B \\
        \partial_A \otimes \id_{H_B} - i\id_{H_A} \otimes \partial_B & 0
    \end{pmatrix}
\end{split}
\end{equation*}
as an operator $({\rm Dom}(\partial_A) \otimes {\rm Dom}(\partial_B))^{\oplus 2} \to (H_A \otimes H_B)^{\oplus 2}$. We let $\partial_A \times \partial_B$ be the closure of this operator.

\medskip \textbf{Odd times even and even times odd:} In both these cases $H = H_A \otimes H_B$ without any grading. If $(B,H_B,\partial_B,\gamma_B)$ is even, we set $\partial_A \times \partial_B$ to be the closure of
\begin{equation*}
    \partial_A \otimes \gamma_B + \id_{H_A} \otimes \partial_B \colon {\rm Dom}(\partial_A) \otimes {\rm Dom}(\partial_B) \to H_A \otimes H_B.
\end{equation*}
If $(A,H_A,\partial_A,\gamma_A)$ is even, we set $\partial_A \times \partial_B$ to be the closure of
\begin{equation*}
     \partial_A \otimes \id_{H_B} + \gamma_A \otimes \partial_B \colon {\rm Dom}(\partial_A) \otimes {\rm Dom}(\partial_B) \to H_A \otimes H_B.
\end{equation*}

The following follows by specializing \cite[Proposition 4.10]{Kaad24} and \cite[Theorem 7.1]{Kaad24} to our setting.

\begin{lemma}\label{thm:spectral triples Kaad}
    Let $(A,H_A,\partial_A)$ and $(B,H_B,\partial_B)$ be unital spectral triples. Then for any $x \in A \otimes_{\rm min} B$ we have that
    \begin{align*}
        (1_A \otimes L_{\partial_B})(x) & \leq L_{\partial_A \times \partial_B}(x), \text{ and} \\
        (L_{\partial_A} \otimes 1_B)(x) & \leq L_{\partial_A \times \partial_B}(x).
    \end{align*}
\end{lemma}
\begin{proof}
    Note that Kaad considers \emph{operator seminorms}, cf.~\cite[Definition 4.1]{Kaad24}. Looking at the $s = 1$ case from \cite[Proposition 4.10]{Kaad24}, and the proof of \cite[Theorem 7.1]{Kaad24}, we get that for any $x \in A \otimes_{\rm min} B$,
    \begin{align*}
        (L_{\partial_A} \otimes 1_B)(x)
        & = \sup\{ L_{\partial_A}((\id_{A} \otimes \psi)(x)) : \psi \in S(B) \} \\
        & \leq \sup\{ L_{\partial_{A}^{\oplus n}}((\id_{M_n(A)} \otimes \psi)(x)) : n \geq 1, \ \psi \in \UCP(B,M_n(\bbC)) \} \\
        & = L_{\partial_A \widehat{\otimes} 1_B}(x) \quad\quad\quad \text{by \cite[Proposition 4.10]{Kaad24}} \\
        & \leq L_{\partial_A \times \partial_B}(x) \quad\quad\quad \text{by the proof of \cite[Theorem 7.1]{Kaad24}}.
    \end{align*}
    The case that $(1_A \otimes L_{\partial_B}) \leq L_{\partial_A \times \partial_B}$ is completely symmetrical to what we have presented here.
\end{proof}

The next theorem is the main result of this section. 
In brief, it tells us that sequences of stable metrics not only exist in theory, but in fact appear naturally using constructions native to noncommutative geometry. 
Specifically, we obtain such sequences using seminorms arising from iteratively applying the external Kasparov product of spectral triples. 

The stability of the sequence $\{\Delta_{\Tr_n \otimes \tau, L_n}^{\rm min}\}_{n \in \bbN}$ in \cref{thm:stability 2} is completely independent of choice of spectral triple for $M_n(\bbC)$, $n \in \bbN$. Physically, this is precisely the behavior we want as the different matrix amplifications model interactions with an ancillary system, whose internal structure should not affect quantum processes in the system under consideration. It is then altogether more surprising that this property is guaranteed by the external Kasparov product. 

In the following, we will for ease of notation write $1$ instead of $\id_H$ for the respective identity maps appearing in the formulas for $\partial_A \times \partial_B$.

\begin{theorem}\label{thm:stability 2}
    Let $(A,H_A,\partial_A)$, and $(B^{\rm op},H_B,\partial_B)$ be unital spectral triples, and assume that $B$ admits an amenable faithful trace $\tau$. For each $n \in \bbN$, let $(M_n(\bbC),H_n,\partial_n)$ be a spectral triple. By identifying $M_n(\bbC)^{\rm op}$ with $M_n(\bbC)$ we may view $(M_n(\bbC),H_n,\partial_n)$ as a spectral triple for $M_n(\bbC)^{\rm op}$ as well. Let
    \begin{align*}
        L_1 & = L_{\partial_A \times \partial_B}, \text{ and } \\
        L_n & = L_{(\partial_n \times \partial_n) \times (\partial_A \times \partial_B)} \circ \Sigma_{[23]}, \text{ for } n \geq 2.
    \end{align*}
    Then the sequence $\{\Delta_{\Tr_n \otimes \tau, L_n}^{\rm min}\}_{n \in \bbN}$ of extended metrics on $\TC_{\Tr_n \otimes \tau}(M_n(\bbC) \otimes A, M_n(\bbC) \otimes B)$ is stable.
\end{theorem}
\begin{proof}
    For ease of notation, let $\partial = (\partial_n \times \partial_n) \times (\partial_A \times \partial_B)$. To apply \cref{thm:stability 1} and \cref{rem:stability for spatial norm} we need to check that the following two conditions are satisfied:
    \begin{enumerate}
        \item $(1_{M_n(\bbC) \otimes M_n(\bbC)^{\rm op}} \otimes L_1) \circ \Sigma_{[23]} \leq L_n$, and
        \item $L_n(\Sigma_{[23]}(1_n \otimes 1_n^{\rm op} \otimes x)) \leq 1$ for every $x \in A \otimes_{\rm min} B^{\rm op}$ satisfying $L_1(x) \leq 1$.
    \end{enumerate}
    By \cref{thm:spectral triples Kaad} we have that
    \begin{equation*}
        1_{M_n(\bbC) \otimes M_n(\bbC)^{\rm op}} \otimes L_{\partial_A \times \partial_B} \leq L_{\partial}.
    \end{equation*}
    Since pre-composing with $\Sigma_{[23]}$ does not change this inequality we get that
    \begin{equation*}
        (1_{M_n(\bbC) \otimes M_n(\bbC)^{\rm op}} \otimes L_1) \circ \Sigma_{[23]} \leq L_n.
    \end{equation*}
    It remains to show that $L_n(\Sigma_{[23]}(1_n \otimes 1_n^{\rm op} \otimes x)) \leq 1$ for every $x \in A \otimes_{\rm min} B^{\rm op}$ satisfying $L_1(x) \leq 1$. By the definition of $L_n$, this means that we want to show that
    \begin{equation*}
        L_{\partial}(1_n \otimes 1_n^{\rm op} \otimes x) \leq 1
    \end{equation*}
    for every $x \in A \otimes_{\rm min} B^{\rm op}$ satisfying $L_{\partial_A \times \partial_B}(x) \leq 1$. We will in fact prove the stronger statement that
    \begin{equation}\label{eqn:seminorm from external products}
        L_{\partial}(1_n \otimes 1_n^{\rm op} \otimes x) = L_{\partial_A \times \partial_B}(x)
    \end{equation}
    for all $x \in A \otimes_{\rm min} B^{\rm op}$. To prove this we only need to consider the parity of the spectral triple for $M_n(\bbC) \otimes M_n(\bbC)^{\rm op}$, and the parity of the spectral triple for $A \otimes_{\rm min} B^{\rm op}$. In fact, since we are considering the same spectral triple for $M_n(\bbC)$ and $M_n(\bbC)^{\rm op}$, the resulting external product of these two spectral triples will always be even. Hence, we only need to consider the parity of the spectral triple for $A \otimes_{\rm min} B^{\rm op}$. In both cases we see that
    \begin{equation*}
        \partial
        = (\partial_n \times \partial_n) \otimes 1
        + \gamma_{M_n(\bbC) \otimes M_n(\bbC)^{\rm op}} \otimes (\partial_A \times \partial_B).
    \end{equation*}
    Hence, for any $x \in A \otimes_{\rm min} B^{\rm op}$ we have that
    \begin{align*}
        L_\partial(1_n \otimes 1_n^{\rm op} \otimes x)
        & = \|[\partial, 1_n \otimes 1_n^{\rm op} \otimes x]\| \\
        & = \|[(\partial_n \times \partial_n) \otimes 1 + \gamma_{M_n(\bbC) \otimes M_n(\bbC)^{\rm op}} \otimes (\partial_A \times \partial_B), 1_n \otimes 1_n^{\rm op} \otimes x]\| \\
        & = \|[\gamma_{M_n(\bbC) \otimes M_n(\bbC)^{\rm op}} \otimes (\partial_A \times \partial_B), 1_n \otimes 1_n^{\rm op} \otimes x]\| \\
        & = \|\gamma_{M_n(\bbC) \otimes M_n(\bbC)^{\rm op}} \otimes [\partial_A \times \partial_B, x]\| \\
        & = \|\gamma_{M_n(\bbC) \otimes M_n(\bbC)^{\rm op}}\| \|[\partial_A \times \partial_B, x]\| \\
        & = \| [\partial_A \times \partial_B, x]\| \\
        & = L_{\partial_A \times \partial_B}(x). \qedhere
    \end{align*}
\end{proof}

\begin{remark}
    Note that one could use two different spectral triples for $M_n(\bbC)$ and for $M_n(\bbC)^{\rm op}$ in \cref{thm:stability 2}. If one chooses to do this in a way such that these spectral triples have different parities then one would have to check that equation \eqref{eqn:seminorm from external products} still holds in this case.
\end{remark}

\subsection{Amenable groups and twisted group \texorpdfstring{$C^*$}{C*}-algebras}\label{sec:group-Cstar-algebras-and-length-functions}

By specializing to countable amenable discrete groups, we will in this section construct a stable sequence of extended metrics on the set of trace channels on a twisted group $C^*$-algebra, see \cref{example:group-example-chaining-and-stability}. Moreover, we will show that the class of unital completely positive maps on twisted group $C^*$-algebras coming from the normalized positive definite functions on the group behave well with respect to the chaining property for the same metrics, see \cref{thm:chaining 3}. We use the notation introduced in \cref{twisted}.

\begin{lemma}\label{lemma:adjointability of multipliers}
    Let $G$ be a discrete group, $\sigma\in Z^2(G,\bbT)$, and $\phi \in P(G)$. Then the completely positive map $M_\phi \colon \Cred(G,\sigma) \to \Cred(G,\sigma)$ is $(\tau_{\sigma},\tau_{\sigma})$-adjointable, and $M_\phi^\sharp = M_{\phi^\circ}$ where the function $\phi^\circ \in P(G)$ is given by
    \begin{align*}
        \phi^\circ (g) = \phi(g^{-1})
    \end{align*}
    for all $g \in G$.
\end{lemma}
\begin{proof}
    That $\phi^\circ \in P(G)$ follows by \cite[p.~7]{BedosConti2016},  
    so 
    $M_{\phi^\circ} \colon \Cred(G,\sigma) \to \Cred(G, \sigma)$ is completely positive. 
    For $f_1 , f_2\in C_c(G,\sigma)$, we get that
    \begin{align*}
        \tau_{\sigma}(M_\phi(f_1) f_2) & = \sum_{g \in G} \phi(g) f_1(g) f_2(g^{-1}) \sigma(g,g^{-1}) \\
        & = \sum_{g \in G} f_1(g) \phi^\circ(g^{-1}) f_2(g^{-1}) \sigma(g,g^{-1}) \\
        & = \tau_{\sigma}(f_1 M_{\phi^\circ}(f_2)).
    \end{align*}
    By continuity, this implies that $M_\phi$ is $(\tau_{\sigma},\tau_{\sigma})$-adjointable, with $M_\phi^\sharp = M_{\phi^\circ}$.
\end{proof}

\begin{remark}\label{rem:unitality of multipliers}
    It is clear that $\phi(1_G) = \phi^\circ(1_G)$. Hence, $\phi \in P_1(G)$ if and only if $\phi^\circ \in P_1(G)$.
\end{remark}

\begin{proposition}\label{prop:multipliers as self-adjoint subset of ucp maps}
    Let $G$ be a discrete group, and $\sigma \in Z^2(G,\bbT)$.
    Then $P_{1}(G)$ can be identified with a  subset of $\TC_{\tau_\sigma}(\Cred(G,\sigma),\Cred(G,\sigma))$ via the embedding $\phi \mapsto M_\phi$. Moreover, we have $M_\phi \in \TC_{\tau_\sigma}(\Cred(G,\sigma),\Cred(G,\sigma))$ if and only if $M_\phi^\sharp = M_{\phi^\circ} \in TC_{\tau_\sigma}(\Cred(G,\sigma),\Cred(G,\sigma))$.
\end{proposition}
\begin{proof}
    Let $\phi \in P_{1}(G)$. Since $M_\phi \in \UCP(\Cred(G,\sigma),\Cred(G,\sigma))$, we have in particular that $M_\phi \in \TC_{\tau_\sigma}(\Cred(G,\sigma),\Cred(G,\sigma))$. The result now follows by \cref{lemma:adjointability of multipliers} and \cref{rem:unitality of multipliers}.
\end{proof}

\begin{theorem}\label{thm:chaining 3}
    Let $G$ be a countable discrete amenable group, $l \colon G \to [0,\infty)$ be a proper length function, and $\sigma \in Z^2(G,\bbT)$. Then the metric $\Delta_{\tau_\sigma,L_{\partial_l \times \partial_l}}^{\rm min}$ induced from the external Kasparov product of $(\Cred(G,\sigma), \ell^2(G), \partial_l)$ with $(\Cred(G,\sigma)^{\rm op}, \ell^2(G), \partial_l)$ restricted to $P_{1}(G)$ satisfies chaining. That is, we have
    \begin{align*}
       \Delta_{\tau_\sigma,L_{\partial_l \times \partial_l}}^{\rm min} (M_{\phi_{1}} \circ M_{\phi_{2}}, M_{\phi_{3}} \circ M_{\phi_{4}})
       \leq \Delta_{\tau_\sigma,L_{\partial_l \times \partial_l}}^{\rm min}(M_{\phi_{1}}, M_{\phi_{3}}) + \Delta_{\tau_\sigma,L_{\partial_l \times \partial_l}}^{\rm min}(M_{\phi_{2}}, M_{\phi_{4}})
    \end{align*}
    for all $\phi_1, \phi_2, \phi_3, \phi_4 \in P_1(G)$.
\end{theorem}
\begin{proof}
    By \cite[Corollary 3.8]{Bedos1995} we have that if $G$ is amenable then $\tau_{\sigma}$ is an amenable tracial state, hence the metric $\Delta_{\tau_\sigma,L_{\partial_l \times \partial_l}}^{\rm min}$ is well-defined. To apply \cref{thm:chaining 1} and \cref{rem:Chaining 1 for spatial norm} we only need to verify that
    \begin{align*}
        L_{\partial_l \times \partial_l} \circ (M_\phi \otimes \id_{\Cred(G,\sigma)^{\rm op}})) & \leq L_{\partial_l \times \partial_l}, \text{ and } \\
        L_{\partial_l \times \partial_l}^{\rm op} \circ (M_\phi^\sharp \otimes \id_{\Cred(G,\sigma)^{\rm op}})) & \leq L_{\partial_l \times \partial_l}^{\rm op}
    \end{align*}
    for every $\phi \in P_1(G)$. Let $x \in \Cred(G,\sigma) \otimes_{\rm min} \Cred(G,\sigma)^{\rm op}$ and $\phi \in P_1(G)$. As in the proof of \cref{lemma:L-cont-multipliers}, we compute that
    \begin{align*}
        & \left[
        \begin{pmatrix}
            0 & \partial_l \otimes 1 + i \otimes \partial_l \\
            \partial_l \otimes 1 - i \otimes \partial_l & 0
        \end{pmatrix},
        \begin{pmatrix}
            (M_\phi \otimes \id_{\Cred(G,\sigma)^{\rm op}})(x) & 0 \\
            0 & (M_\phi \otimes \id_{\Cred(G,\sigma)^{\rm op}})(x)
        \end{pmatrix}
        \right] & \\
        & =
        \begin{pmatrix}
            0 & [(\partial_l \otimes 1 + i \otimes \partial_l), (M_\phi \otimes \id_{\Cred(G,\sigma)^{\rm op}})(x)] \\
            [(\partial_l \otimes 1 - i \otimes \partial_l), (M_\phi \otimes \id_{\Cred(G,\sigma)^{\rm op}})(x)] & 0
        \end{pmatrix} & \\
        & =
        \begin{pmatrix}
            0 & (T^\phi \otimes \id_{\Cred(G,\sigma)^{\rm op}})[(\partial_l \otimes 1 + i \otimes \partial_l), x] \\
            (T^\phi \otimes \id_{\Cred(G,\sigma)^{\rm op}})[(\partial_l \otimes 1 - i \otimes \partial_l), x] & 0
        \end{pmatrix} \\
        & =
        \begin{pmatrix}
            (T^\phi \otimes \id_{\Cred(G,\sigma)^{\rm op}}) & 0 \\
            0 & (T^\phi \otimes \id_{\Cred(G,\sigma)^{\rm op}})
        \end{pmatrix}
        \begin{pmatrix}
            0 & [(\partial_l \otimes 1 + i \otimes \partial_l), x] \\
            [(\partial_l \otimes 1 - i \otimes \partial_l), x] & 0
        \end{pmatrix}
    \end{align*}
    Computing the operator norm of this proves that $L_{\partial_l \times \partial_l} \circ (M_\phi \otimes \id_{\Cred(G,\sigma)^{\rm op}}) \leq L_{\partial_l \times \partial_l}$, since
    \begin{equation*}
        \left\|
        \begin{pmatrix}
            (T^\phi \otimes \id_{\Cred(G,\sigma)^{\rm op}}) & 0 \\
            0 & (T^\phi \otimes \id_{\Cred(G,\sigma)^{\rm op}})
        \end{pmatrix}
        \right\|
        = 1.
    \end{equation*}

    For $\phi \in P_1(G)$ define $M_\phi^{\rm op} \colon \Cred(G,\sigma)^{\rm op} \to \Cred(G,\sigma)^{\rm op}$ by
    \begin{equation*}
        M_\phi^{\rm op}(b^{\rm op}) = M_\phi(b)^{\rm op}.
    \end{equation*}
    Note that for any $x \in \Cred(G,\sigma) \otimes_{\rm min} \Cred(G,\sigma)^{\rm op}$ we then have that
    \begin{align*}
        \Sigma^{\rm op} ((M_\phi \otimes \id_{\Cred(G,\sigma)^{\rm op}})(x)) = (\id_{\Cred(G,\sigma)^{\rm op}} \otimes M_\phi^{\rm op})(\Sigma^{\rm op}(x)).
    \end{align*}
    We furthermore have by \cref{lemma:adjointability of multipliers} that $M_\phi^\sharp = M_{\phi^\circ}$. Hence it suffices to show that
    \begin{equation*}
        L_{\partial_l \times \partial_l} \circ (\id_{\Cred(G,\sigma)^{\rm op}} \otimes M_\phi^{\rm op}) \leq L_{\partial_l \times \partial_l}
    \end{equation*}
    for all $\phi \in P_1(G)$. This follows by the same computation as above.

    With this we may appeal to \cref{thm:chaining 1} and conclude that the metric $\Delta_{\tau_\sigma,L_{\partial_l \times \partial_l}}^{\rm min}$ satisfies chaining when restricted to $P_{1}(G)$, which is what we wanted to show.
\end{proof}

\begin{remark}
    The metric $\Delta_{\tau_\sigma,L_{\partial_l \times \partial_l}}^{\rm min}$ is defined on the whole of $\TC_{\tau_\sigma}(\Cred(G,\sigma),\Cred(G,\sigma))$, of which we view $P_1(G)$ as a subspace. If one is only interested in defining a metric on $P_1(G)$ one could instead induce such a metric by considering the correspondence $P_1(G) \simeq S(C^*(G)) \simeq S(\Cred(G))$, where the latter isomorphism holds whenever $G$ is amenable. In this case we similarly get a unital spectral triple $(\Cred(G),\ell^2(G),\partial_l)$ from which we get a metric $\mk_{L_{\partial_l}}$ on $P_1(G)$. We are however interested in determining which trace channels satisfy chaining, and as such the metric $\mk_{L_{\partial_l}}$ on $P_1(G)$ does not suffice as it is not clear if such a metric extends to $\TC_{\tau_\sigma}(\Cred(G,\sigma),\Cred(G,\sigma))$. 
\end{remark}

Combining \cref{thm:stability 2} and \cref{thm:chaining 3} gives many examples of stable sequences of extended metrics that satisfy chaining. We record the following example explicitly.

\begin{example}\label{example:group-example-chaining-and-stability}
    Let $G$ be a countable amenable discrete group, $\sigma \in Z^2 (G, \bbT)$, and $l \colon G \to [0,\infty)$ be a proper length function. This yields spectral triples $(C^*_r(G,\sigma), \ell^2(G), \partial_l)$ and $(\Cred(G,\sigma)^{\rm op}, \ell^2(G), \partial_l)$ as we have seen above. As $G$ is amenable, $C^*_r (G,\sigma) $ is a nuclear $C^*$-algebra. By \cref{thm:stability 2} we will obtain a stable sequence of extended metrics $\{\Delta_{\Tr_n \otimes \tau, L_n}^{\rm min}\}_{n \in \bbN}$ after choosing a spectral triple $(M_n (\bbC), H_n ,\partial_n)$ for each $n \in \bbN$. Such spectral triples have been studied in \cite{DAndreaMartinetti2021}, where they show how to recover the matricial Wasserstein-1 distance from \cite{ChenGeorgiouNingTannenbaum2017} from the point of view of noncommutative geometry. Namely, pick $N \in \bbN$ and,  for each $n\in \bbN$,  set $H_n = \bbC^n \otimes \bbC^N$ and
    \begin{align*}
        \partial_n = \sum_{i=1}^N R_i \otimes e_{i,i}^{(N)}
    \end{align*}
    where $R_i \in M_n (\bbC)$ is a self-adjoint matrix for each $i = 1, \ldots, N$. Letting $M_n (\bbC)$ act on the first factor of $H_n$, we then have a spectral triple $(M_n(\bbC), H_n, \partial_n)$ for each $n \in \bbN$. Recall that for any $\varphi \in S(M_n (\bbC))$, there is $\rho \in M_n(\bbC)$ such that $\varphi = \Tr_n(\rho\, \cdot)$.  D'Andrea and Martinetti show in \cite[Proposition 5]{DAndreaMartinetti2021} that the associated Monge-Kantorovi\v{c} metric can be realized as
    \begin{align*}
        \mk_{\partial_n}(\varphi_1, \varphi_2) = \inf \big\{  \Tr_n (\sqrt{u^* u} ) \mid u_i \in M_n(\bbC), \sum_{i=1}^N [R_i, u_i] = \rho_1 - \rho_2 \big\}
    \end{align*}
    where $\varphi_j = \Tr_n(\rho_j \cdot )$ and $u = (u_1, \ldots , u_N)$. These spectral triples are therefore related to the trace norm. For each $n$, make such a choice of spectral triple $(M_n(\bbC), H_n, \partial_n)$. We then obtain a sequence of seminorms $L_n = L_{(\partial_n \times \partial_n) \times (\partial_l \times \partial_l)} \circ \Sigma_{[23]}$ by taking the appropriate Kasparov products as in \cref{thm:stability 2}, which yield a stable sequence of metrics $\{\Delta_{\Tr_n \otimes \tau, L_n}^{\rm min}\}_{n \in \bbN}$, where each $\Delta_{\Tr_n \otimes \tau, L_n}^{\rm min}$ is a metric on $\TC_{\Tr_n \otimes \tau}(M_n(\bbC) \otimes C^*_r(G, \sigma), M_n(\bbC) \otimes C^*_r(G, \sigma))$. By  \cref{thm:chaining 3}, the metric $\Delta_{\Tr_1 \otimes \tau_\sigma, L_1}^{\rm min} = \Delta_{\tau_\sigma,L_{\partial_l \times \partial_l}}^{\rm min}$ satisfies chaining with respect to the unital completely positive maps arising as multipliers $M_\phi$ for $\phi \in P_1 (G)$.
\end{example}

\subsection*{Data availability}

Data sharing is not applicable to this article as no datasets were generated or analyzed during the current study.

\subsection*{Conflict of interest}

The authors have no conflicts of interest to declare that are relevant to the content of this article.

\printbibliography

\end{document}